\documentclass[10pt,twoside]{article}
\def\YEAR{\year}\newcount\VOL\VOL=\YEAR\advance\VOL by-1995
\def\firstpage{1}\def\lastpage{1000}
\def\received{}\def\revised{}
\def\communicated{}

\makeatletter
\def\magnification{\afterassignment\m@g\count@}
\def\m@g{\mag=\count@\hsize6.5truein\vsize8.9truein\dimen\footins8truein}
\makeatother

\font\eightrm=cmr8
\font\caps=cmcsc10                    
\font\Caps=cmcsc10 scaled \magstep1   

\makeatletter
\@specialpagefalse\headheight=8.5pt
\def\DocMath{}
\renewcommand{\@evenhead}{%
    \ifnum\thepage>\lastpage\rlap{\thepage}\hfill%
    \else\rlap{\thepage}\slshape\leftmark\hfill{\caps\SAuthor}\hfill\fi}%
\renewcommand{\@oddhead}{%
    \ifnum\thepage=\firstpage{\DocMath\hfill\llap{\thepage}}%
    \else{\slshape\rightmark}\hfill{\caps\STitle}\hfill\llap{\thepage}\fi}%
\makeatother

\def\TSkip{\bigskip}
\newbox\TheTitle{\obeylines\gdef\GetTitle #1
\ShortTitle  #2
\SubTitle    #3
\Author      #4
\ShortAuthor #5
\EndTitle
{\setbox\TheTitle=\vbox{\baselineskip=20pt\let\par=\cr\obeylines%
\halign{\centerline{\Caps##}\cr\noalign{\medskip}\cr#1\cr}}%
	\copy\TheTitle\TSkip\TSkip%
\def\next{#2}\ifx\next\empty\gdef\STitle{#1}\else\gdef\STitle{#2}\fi%
\def\next{#3}\ifx\next\empty%
    \else\setbox\TheTitle=\vbox{\baselineskip=20pt\let\par=\cr\obeylines%
    \halign{\centerline{\caps##} #3\cr}}\copy\TheTitle\TSkip\TSkip\fi%
\centerline{\caps #4}\TSkip\TSkip%
\def\next{#5}\ifx\next\empty\gdef\SAuthor{#4}\else\gdef\SAuthor{#5}\fi%
\ifx\received\empty\relax
    \else\centerline{\eightrm Received: \received}\fi%
\ifx\revised\empty\TSkip%
    \else\centerline{\eightrm Revised: \revised}\TSkip\fi%
\ifx\communicated\empty\relax
    \else\centerline{\eightrm Communicated by \communicated}\fi\TSkip\TSkip%
\catcode'015=5}}\def\Title{\obeylines\GetTitle}
\def\Abstract{\begingroup\narrower
    \parskip=\medskipamount\parindent=0pt{\caps Abstract. }}
\def\EndAbstract{\par\endgroup\TSkip}

\long\def\MSC#1\EndMSC{\def\arg{#1}\ifx\arg\empty\relax\else
     {\par\narrower\noindent%
     2000 Mathematics Subject Classification: #1\par}\fi}

\long\def\KEY#1\EndKEY{\def\arg{#1}\ifx\arg\empty\relax\else
	{\par\narrower\noindent Keywords and Phrases: #1\par}\fi\TSkip}

\newbox\TheAdd\def\Addresses{\vfill\copy\TheAdd\vfill
    \ifodd\number\lastpage\vfill\eject\phantom{.}\vfill\eject\fi}
{\obeylines\gdef\GetAddress #1
\Address #2 
\Address #3
\Address #4
\EndAddress
{\def\xs{4.3truecm}\parindent=0pt
\setbox0=\vtop{{\obeylines\hsize=\xs#1\par}}\def\next{#2}
\ifx\next\empty 
     \setbox\TheAdd=\hbox to\hsize{\hfill\copy0\hfill}
\else\setbox1=\vtop{{\obeylines\hsize=\xs#2\par}}\def\next{#3}
\ifx\next\empty 
     \setbox\TheAdd=\hbox to\hsize{\hfill\copy0\hfill\copy1\hfill}
\else\setbox2=\vtop{{\obeylines\hsize=\xs#3\par}}\def\next{#4}
\ifx\next\empty\ 
     \setbox\TheAdd=\vtop{\hbox to\hsize{\hfill\copy0\hfill\copy1\hfill}
                \vskip20pt\hbox to\hsize{\hfill\copy2\hfill}}
\else\setbox3=\vtop{{\obeylines\hsize=\xs#4\par}}
     \setbox\TheAdd=\vtop{\hbox to\hsize{\hfill\copy0\hfill\copy1\hfill}
	        \vskip20pt\hbox to\hsize{\hfill\copy2\hfill\copy3\hfill}}
\fi\fi\fi\catcode'015=5}}\gdef\Address{\obeylines\GetAddress}

\usepackage{amsfonts,amsmath,amsthm,mathrsfs,amssymb,graphicx,multirow,latexsym}

\usepackage[all]{xy}

\usepackage{hyperref}

\usepackage{tikz}
\usetikzlibrary{patterns}
\usepackage{graphics}
\usepackage{cancel}
\usepackage{verbatim}

\usepackage{enumerate} 
\usepackage{color} 
\usepackage{fancybox}

\newtheorem{thm}{Theorem}[section]
\newtheorem{cor}[thm]{Corollary}
\newtheorem{lem}[thm]{Lemma}

\newtheorem{dff}[thm]{Definition}

\newtheorem{xmp}[thm]{Example}

\newtheorem{remark}[thm]{Remark}

\def\mbf{\mathbf}
\def\W{\mathbf{W}}

\def\tn{\textnormal}
\def\mcal{\mathcal}
\def\goth{\mathfrak}
\def\Map{\textnormal{Map}}

\newcommand{\PGL}{\textnormal{PGL}}
\newcommand{\Lev}{\textnormal{Lev}}

\def\QQ{\mathbf{Q}}

\def\ZZ{\mathbf{Z}}

\begin{document}
\Title On the structure of Witt-Burnside rings attached to pro-$p$ groups
\ShortTitle Witt-Burnside rings attached to pro-$p$ groups
\SubTitle   
\Author  Lance Edward Miller
\ShortAuthor  L. E. Miller
\EndTitle
\Abstract 
The $p$-typical Witt vectors are a ubiquitous object in algebra and number theory. They arise as a functorial construction that takes perfect fields $k$ of prime characteristic $p > 0$ to $p$-adically complete discrete valuation rings of characteristic $0$ with residue field $k$ and are universal in that sense. A. Dress and C. Siebeneicher generalized this construction by producing a functor $\W_G$ attached to any profinite group $G$. The $p$-typical Witt vectors arise as those attached to the $p$-adic integers. Here we examine the ring structure of $\W_G(k)$ for several examples of pro-$p$ groups $G$ and fields $k$ of characteristic $p$. We will show that the structure is surprisingly more complicated than the $p$-typical case. 
\EndAbstract
\MSC 
\EndMSC
\KEY 
\EndKEY
\Address University of Utah \\ Department of Mathematics  \\ 155 S 1400 E Room 233 \\ Salt Lake City, UT \\ 84112-0090
\Address
\Address
\Address
\EndAddress

\section[Introduction]{Introduction}

The purpose of this article is to explore what kinds of rings lie in the image of the functors $\W_G$ introduced by A.~Dress and C.~Siebeneicher \cite{DS} in the case that $G$ is a pro-$p$ group for prime integer $p$. 
These functors were originally defined for all profinite groups and are now called Witt-Burnside functors due to the fact that they generalize both the $p$-typical (recovered when $G = \ZZ_p$ as an additive group) and `big' Witt vector construction (recovered when $G = \widehat{\ZZ}$), as well as Burnside functors (recovered as $\W_G(\ZZ)$). We call rings lying in the image of a Witt-Burnside functor {\it Witt-Burnside rings}. 
These functors are of significant interest, and yet to date the types of rings which are produced from this construction lack description. To illustrate some of their recent importance, we remark that Witt-Burnside functors have been used extensively to study equivariant ring spectra as they arise as a left adjoint for Tambara functors \cite{Bru05,Bru07,Str}.

Many constructions of Witt-Burnside functors have been given. Specifically, J.~Graham's construction utilizes ring valued $G$-sets \cite{Graham}. J.~Elliott  gave a unified construction relating the Graham and Dress and Siebeneicher approaches \cite{Elliott}. Y.~Oh has studied some decomposition and $q$-deformation questions \cite{OH1,OH2, OH12}. Beyond the classical $G = \ZZ_p$ case, not much about the ring structure of the images of Witt-Burnside functors is known. Motivated by the extensive applications enjoyed by the $p$-typical and big Witt vectors, this paper addresses structural questions about Witt-Burnside rings under the assumptions that $G$ is an infinite pro-$p$ group and $k$ is a  field of characteristic $p > 0$.

The construction of the classical (both $p$-typical and big) Witt vectors uses the Witt polynomials to define certain 
addition and multiplication polynomials with rational coefficients which do not obviously have integral coefficients \cite{Witt}. 
A.~Dress and C.~Siebeneicher generalized the Witt polynomials to a family of multivariable polynomials associated to any profinite group $G$. Like classical Witt polynomials, these polynomials obviously have $\mbf{Q}$-coefficients and a significant theorem of Dress and Siebeneicher shows that they in fact have integral coefficients. Thus one can use these polynomials, in an analogous way to the construction of classical Witt vectors, to define a functor $\W_G$ on the category of commutative rings for each profinite group $G$. For infinite pro-$p$ groups $G$, the functor $\W_G$ retains the surprising property of taking rings of characteristic $p$ to rings of characteristic $0$. 

For perfect fields $k$ of characteristic $p$, the ring $\W_{\ZZ_p}(k)$ is the ring of classical $p$-typical Witt vectors. These are  $p$-adically complete discrete valuation domains with maximal ideal $(p)$ and residue field $k$. From \cite[Thm. 2]{DS} it is known that $\W_G(k)$ is a ring of characteristic zero when $G \not\cong \ZZ_p$ is an infinite pro-$p$ group and $k$ is a field. We show as expected that $\W_G(k)$ shares some properties with the $p$-typical case when $G$ is pro-$p$ and $k$ has characteristic $p$. 

\

{\bf{Theorem} (Cf., Theorem~\ref{units})}
{\it{For $G$ a pro-$p$ group and $A$ a local ring of characteristic $p > 0$, $\W_G(A)$ is a local ring.}}

\

However, the similarities do not run deep! 

\

{\bf{Theorem} (Cf., Theorem~\ref{thm:Zpsqnotnoth})} 
{\it{For $G = \ZZ_p^d$, $k$ a field of characteristic $p > 0$ and $d \geq 2$, the maximal ideal of $\W_G(k)$ is not finitely generated, so $\W_{G}(k)$ is not Noetherian.}}

\

The core reason for the complication in the case $G = \ZZ_p^d$ when $d \geq 2$ versus $d=1$ is the existence of more than one maximal subgroup. In general, when $G$ is any pro-$p$ group that is not pro-cyclic there is more than one maximal open subgroup $H$, necessarily normal. These subgroups describe certain coordinates in the Witt vectors attached to $G$ for which sums and products have repeated values, or redundencies. In particular, $G = \ZZ_p^d$ has a fairly homogeneous subgroup structure as every open subgroup is (non-canonically) isomorphic to $G$. This means the redundancy behavior when $d \geq 2$ propagates and manifests as a much smaller square of the maximal ideal than expected. Another consequence of there being more than one maximal subgroup of an infinite pro-$p$ group $G$ when $G \not\cong \ZZ_p$ is the ability to construct zero divisors. This was essentially known in \cite{DS}. Our second main goal is to give some control on the zero divisors of $\W_G(k)$ for $G = \ZZ_p^2$ and $k$ a field of characteristic $p$. 

\

{\bf{Theorem} (Cf., Theorem~\ref{thm:Red})}
{\it{For $G = \ZZ_p^2$ and any field $k$ of characteristic $p$, the ring $\W_{G}(k)$ is reduced.}}

\

The methods used here fail for $d > 2$ due to the reliance on a certain property of the subgroup structure of $\ZZ_p^2$ which is not satisfied more generally. 

The rest of the paper is organized as follows. Section \ref{sec:prelim} develops the preliminary definitions, constructions, and facts about Witt-Burnside rings. We also prove necessary lemmas for the paper and close with the proof that $\W_G(A)$ is local when $G$ is pro-$p$ and $A$ is a local ring of characteristic $p$. Section \ref{secd2} discusses in detail the frame of $\ZZ_p^d$ for $d \geq 2$; these groups form the basic examples of the paper. Section \ref{sec:Zpdnotnoth} discusses the failure of finite generation in the maximal ideal $\W_{\ZZ_p^d}(k)$ for $d \geq 2$ when the characteristic of $k$ is $p$. Finally, Section \ref{sec:reduced} concerns nilpotent elements in $\W_G(k)$. Unless otherwise stated, $G$ will always be a profinite group and $p$ denotes  a prime integer. 

\section{Preliminaries}
\label{sec:prelim}

Witt-Burnside rings are constructed utilizing generalized Witt polynomials associated to a profinite group $G$. The index set of these generalized polynomials is the set of isomorphism classes of discrete finite transitive $G$-sets, called the {\it{frame}} of $G$ and denoted $\mcal{F}(G)$. For example, $\mcal{F}(\ZZ_p) = \mbf{N}$ by the correspondence $\ZZ_p/p^n\ZZ_p \leftrightarrow n$. There is a natural partial ordering on $\mcal{F}(G)$. For $T$ and $U$ in $\mcal{F}(G)$ we say $U \leq T$ if there is a $G$-map from $T$ to $U$. Denote the set of all $G$-maps from $T$ to $U$ as $\Map_G(T,U)$ and the number of $G$-maps $\# \Map_G(T,U)$ by $\varphi_T(U)$. Thus $\varphi_T(U) \neq 0$ if and only if $T \leq U$. We summarize some facts about $\mcal{F}(G)$ 

\begin{enumerate}
\item If $T$ and $U$ in $\mcal{F}(G)$ with $U \leq T$, then $\# U$ divides $\# T$ and $\# T / \# U$ represents the size of any of the fibers of any element of $\Map_G(T,U)$. 
\item If the stabilizer subgroups of the points in $T$ are all equal (we will say in this case that $T$ has normal stabilizers or that $T$ is a normal $G$-set), then $\varphi_T(U) = \# U$ for $U \leq T$. 
\item For each $T$ in $\mcal{F}(G)$, there are only finitely many $U$ in $\mcal{F}(G)$ with $U \leq T$. 
\end{enumerate}

The elements of $\mcal{F}(G)$ have a concrete description. Every finite transitive $G$-set $T$ is isomorphic to some coset space $G/H$ with left $G$-action, where $H$ is an open subgroup of $G$ that can be chosen as the stabilizer subgroup of any point in $T$.  
The partial order $\leq$ on coset spaces (considered as $G$-sets up to isomorphism) can be described concretely by $G/K \leq G/H$ if and only if $H$ is conjugate to a subgroup of $K$ (or equivalently, $H$ is a subgroup of a conjugate of $K$). 

For $T \in \mcal{F}(G)$, 
define the $T$-th {\it Witt polynomial} to be 
\begin{equation}\label{Wdef}
W_T(\{X_U\}_{U \in \mcal{F}(G)}) = \sum\limits_{U \leq T} \varphi_T(U) X_U^{\#T / \#U} = X_0^{\#T} + \ldots + \varphi_T(T) X_T,
\end{equation}
where $0$ denotes the trivial $G$-set $G/G$. Trivially $\varphi_T(0) = 1$ for all $T$ in $\mcal{F}(G)$. 
This is a finite sum since there are only finitely many 
$U \leq T$.

For instance, if $G = \mbf{Z}_p$ then the finite transitive $G$-sets up to isomorphism 
are $\mbf{Z}_p/p^n\mbf{Z}_p$ for $n \geq 0$ and the Witt polynomial associated to $\mbf{Z}_p/p^n\mbf{Z}_p$ is the classical $n$-th $p$-typical Witt polynomial. Figure \ref{fig:frame} displays the frame of $\ZZ_2^2$. When $G = \ZZ_p^2$, all $G$-sets have $p+1$ covers (that is, there are exactly $p+1$ $G$-sets in the frame lying immediately above each $G$-set). Other than the trivial $G$-set, each $G$-set below the horizontal line in Figure \ref{fig:frame} has $p$ covers also below the horizontal line. This line is not part of the frame and depicts a property about the stabilizers of various $G$-sets described in Definition~\ref{dff:level}. 

\begin{figure}[htb]
\begin{center}
\scalebox{0.8}{
\begin{tikzpicture}[smooth]


\fill[black] (0,0) circle (0.1cm);

\fill[black] (1,0) circle (0.1cm);
\fill[black] (1,-1) circle (0.1cm);
\fill[black] (1,1) circle (0.1cm);

\fill[black] (3,1.25) circle (0.1cm);
\fill[black] (3,0.75) circle (0.1cm);

\fill[black] (3,0.25) circle (0.1cm);
\fill[black] (3,-0.25) circle (0.1cm);

\fill[black] (3,-0.75) circle (0.1cm);
\fill[black] (3,-1.25) circle (0.1cm);

\fill[black] (3,2.5) circle (0.1cm);

\fill[black] (5,3) circle (0.1cm);
\fill[black] (5,2.5) circle (0.1cm);
\fill[black] (5,2) circle (0.1cm);

\fill[black] (5,-1.2) circle (0.1cm);
\fill[black] (5,-1.45) circle (0.1cm);

\fill[black] (5,-0.65) circle (0.1cm);
\fill[black] (5,-0.9) circle (0.1cm);

\fill[black] (5,-0.15) circle (0.1cm);
\fill[black] (5,-0.4) circle (0.1cm);

\fill[black] (5,0.35) circle (0.1cm);
\fill[black] (5,0.1) circle (0.1cm);

\fill[black] (5,0.85) circle (0.1cm);
\fill[black] (5,0.6) circle (0.1cm);

\fill[black] (5,1.35) circle (0.1cm);
\fill[black] (5,1.1) circle (0.1cm);

\fill[black] (7,0) circle (0.05cm);
\fill[black] (7.2,0) circle (0.05cm);
\fill[black] (7.4,0) circle (0.05cm);

\fill[black] (5.5,2.5) circle (0.05cm);
\fill[black] (5.7,2.5) circle (0.05cm);
\fill[black] (5.9,2.5) circle (0.05cm);

\fill[black] (4,3.25) circle (0.05cm);
\fill[black] (4.2,3.45) circle (0.05cm);
\fill[black] (4.4,3.65) circle (0.05cm);

\draw[-] (0,0) -- (1,0);
\draw[-] (0,0) -- (1,-1);
\draw[-] (0,0) -- (1,1);

\draw[-] (1,1) -- (3,1.25);
\draw[-] (1,1) -- (3,0.75);

\draw[-] (1,0) -- (3,0.25);
\draw[-] (1,0) -- (3,-0.25);

\draw[-] (1,-1) -- (3,-0.75);
\draw[-] (1,-1) -- (3,-1.25);

\draw[-] (-1,1.5) -- (6,1.5);

\draw[-] (1,1) -- (3,2.5);
\draw[-] (1,0) -- (3,2.5);
\draw[-] (1,-1) -- (3,2.5);

\draw[-] (3,2.5) -- (5,3);
\draw[-] (3,2.5) -- (5,2.5);
\draw[-] (3,2.5) -- (5,2);

\draw[-] (3,1.25) -- (5,1.35);
\draw[-] (3,1.25) -- (5,1.1);

\draw[-] (3,0.75) -- (5,0.85);
\draw[-] (3,0.75) -- (5,0.6);

\draw[-] (3,0.25) -- (5,0.35);
\draw[-] (3,0.25) -- (5,0.1);

\draw[-] (3,-0.25) -- (5,-0.15);
\draw[-] (3,-0.25) -- (5,-0.4);

\draw[-] (3,-0.75) -- (5,-0.65);
\draw[-] (3,-0.75) -- (5,-0.9);

\draw[-] (3,-1.25) -- (5,-1.2);
\draw[-] (3,-1.25) -- (5,-1.45);

\draw[-] (3,1.25) -- (5,3);
\draw[-] (3,0.75) -- (5,3);

\draw[-] (3,0.25) -- (5,2.5);
\draw[-] (3,-0.25) -- (5,2.5);

\draw[-] (3,-0.75) -- (5,2);
\draw[-] (3,-1.25) -- (5,2);

\node at (-0.75,0) {$\ZZ_2^2/\ZZ_2^2$} {};
\node at (2,2.5) {$\ZZ_2^2/2 \ZZ_2^2$} {};

%
%
%

\end{tikzpicture}
}
\end{center}
\caption{The frame $\mcal{F}(\ZZ_2^2)$.}
\label{fig:frame}
\end{figure}


\begin{remark}
\label{rmk:trees}
The picture in Figure \ref{fig:frame} is reminiscent of the tree 
of $\ZZ_2$-lattices in $\QQ_2^2$ up to scaling, on which $\PGL_2(\QQ_2)$ acts \cite[p.~71]{TR}. 
However, it is different since  the $G$-sets $\ZZ_2^2/2^r\ZZ_2^2$ appear as separate vertices in Figure \ref{fig:frame}, while the subgroups 
$2^r\ZZ_2^2$ all correspond to the same vertex in the tree for $\PGL_2(\QQ_2)$.
\end{remark}

To simplify notation, write a tuple of variables $X_T$ indexed by all $T$ in $\mathcal{F}(G)$ 
as $\underline{X}$, e.g.,  
$W_T(\{X_U\}_{U \in \mcal{F}(G)}) = W_T(\underline{X})$, 
$\ZZ[\{X_T\}_{T \in \mcal{F}(G)}] = \ZZ[\underline{X}]$, and 

\noindent  $\ZZ[\{X_T, Y_T\}_{T \in \mcal{F}(G)}] = \ZZ[\underline{X}, \underline{Y}]$.  
This underline notation of course depends on $G$.  For any commutative ring $A$, 
a polynomial 
$f(\underline{X}) \in \ZZ[\underline{X}]$ 
defines a function from $\prod_{T \in \mcal{F}(G)} A$ to $A$, and 
for a tuple $\mbf{a} = (a_T)_{T \in \mcal{F}(G)}$ with coordinates in $A$ we write  
$f(\mbf{a}) = f(\{a_T\}_{T \in \mcal{F}(G)}) \in A$.  A similar meaning is 
applied to $f(\mbf{a},\mbf{b})$ for a polynomial 
$f(\underline{X},\underline{Y}) \in \ZZ[\underline{X},\underline{Y}]$.  Generally, we write sequences indexed by $\mcal{F}(G)$ as bold letters $(e.g., \mbf{a},\mbf{b},\mbf{x},\mbf{y},\mbf{v})$ and their $T$-th coordinate is in italics $(e.g., a_T,b_T,x_T,y_T,v_T)$.

Because $X_T$ appears on the right side of (\ref{Wdef}) just in the linear term $\varphi_T(T)X_T$, 
and all variables which appear in other terms are $X_U$ for $U < T$, we get 
the following uniqueness criterion for all the Witt polynomial values together which is equivalent to Lemma 2.1 in \cite[p. 331]{Elliott}.

\begin{thm}\label{invertthm}
If $A$ is a commutative ring which has no $\varphi_T(T)$-torsion, 
then the function $\prod_{T \in \mcal{F}(G)} A \rightarrow 
\prod_{T \in \mcal{F}(G)} A$ given by 
$\mbf{a} \mapsto (W_T(\mbf{a}))_{T \in \mcal{F}(G)}$ is injective. This function is bijective provided each $\varphi_T(T)$ is a unit in $A$. 
\end{thm}

\begin{xmp}
\label{pxmp23}
If $G$ is a pro-$p$ group and $T \cong G/H$, then $\varphi_T(T) = [{\rm{N}}_G(H) : H]$ is a 
power of $p$, so if $p$ is invertible in $A$ then every 
$\mbf{a} \in \prod_{T \in \mcal{F}(G)} A$ 
has the form $(W_T(\mbf{b}))_{T \in \mcal{F}(G)}$ 
for a unique $\mbf{b} \in \prod_{T \in \mcal{F}(G)} A$. 
\end{xmp}

The most important application of Theorem \ref{invertthm} is to the ring 
$A = \QQ[\underline{X},\underline{Y}]$ and the vectors 
$(W_T(\underline{X}) + W_T(\underline{Y}))_{T \in \mcal{F}(G)}$ and 
$(W_T(\underline{X})W_T(\underline{Y}))_{T \in \mcal{F}(G)}$. 
It tells us there are unique families of polynomials $\{S_T(\underline{X},\underline{Y})\}$ and 
$\{M_T(\underline{X},\underline{Y})\}$ in 
$\QQ[\underline{X},\underline{Y}]$ satisfying 
$$
W_T(\underline{X}) + W_T(\underline{Y}) = W_T(\underline{S}) \text{ for all } T \in \mathcal{F}(G)
$$
and 
$$
W_T(\underline{X})W_T(\underline{Y}) = W_T(\underline{M}) \text{ for all } T \in \mathcal{F}(G).
$$
More explicitly, this says 
\begin{equation}\label{STF}
\sum_{U \leq T} \varphi_T(U) X_U^{\# T / \#U} + \sum_{U \leq T} \varphi_T(U) Y_U^{\# T / \#U} = 
\sum_{U \leq T} \varphi_T(U) S_U^{\# T / \#U}
\end{equation}
and
\begin{equation}\label{MTF}
\left(\sum_{U \leq T} \varphi_T(U) X_U^{\# T / \#U}\right)\left(\sum_{U \leq 
T} \varphi_T(U) Y_U^{\# T / \#U}\right) = 
\sum_{U \leq T} \varphi_T(U) M_U^{\# T / \#U}
\end{equation}
for all $T$.  
The polynomials $S_T$ and $M_T$ each only 
depend on the variables $X_U$ and $Y_U$ for $U \leq T$. 

A significant theorem of Dress and Siebeneicher \cite[p.~107]{DS}, which generalizes Witt's theorem 
($G = \ZZ_p$), says that
the polynomials $S_T$ and $M_T$ have coefficients in $\ZZ$. 
We call the $S_T$'s and $M_T$'s the Witt addition and multiplication polynomials, respectively. 
(Obviously they depend on $G$, but that dependence will not be part of the notation). 

\begin{xmp}
\label{xmp:MTST}
Taking $T = 0$, one has 
$$
S_0(\underline{X},\underline{Y}) = X_0 + Y_0 \tn{ and} \
M_0(\underline{X},\underline{Y}) = X_0Y_0.
$$  
If $T \cong G/H$ where $H$ is a maximal open 
subgroup, so $\{U \in \mcal{F}(G) \colon U \leq T\}$ is just $\{0,T\}$ solving for $S_T$ and $M_T$ in 
(\ref{STF}) and (\ref{MTF}) yields
$$
S_T = X_T + Y_T + 
\frac{(X_0+Y_0)^{\#T} - X_0^{\#T} - Y_0^{\#T}}{\varphi_T(T)},$$
$$
M_T = X_0^{\#T}Y_T + X_TY_0^{\#T} + \varphi_T(T)X_TY_T. 
$$
Compare with the first two classical Witt addition and multiplication 
polynomials in \cite[p.~42]{LF}. 
Further addition and multiplication 
polynomials could be very complicated to write out explicitly, as is already apparent 
for the classical Witt vectors if you try to go past the first two polynomials. 
\end{xmp}


Since $S_T$ and $M_T$ have integral coefficients, they can be evaluated on any ring, including 
rings where the hypotheses of Theorem \ref{invertthm} break down, like a ring of characteristic $p$ 
when $G$ is a pro-$p$ group. 

\begin{dff}
Let $G$ be a profinite group. 
For any commutative ring $A$, define the {\it Witt--Burnside ring} $\mbf{W}_G(A)$ to be 
the product space $\prod_{T \in \mathcal{F}(G)} A$ as a set, with 
elements written as $\mbf{a} = (a_T)_{T \in \mathcal{F}(G)}$. 
The ring operations on $\mbf{W}_G(A)$ are 
defined using the Witt addition and multiplication polynomials:
$$
\mbf{a} + \mbf{b} = (S_T(\mbf{a},\mbf{b}))_{T \in \mathcal{F}(G)}
$$
and
$$
\mbf{a} \cdot \mbf{b} = (M_T(\mbf{a},\mbf{b}))_{T \in \mathcal{F}(G)}.
$$
The additive (resp. multiplicative) identity is $(0,0,0,\dots)$ (resp. $(1,0,0,\dots)$). 
\end{dff}

We remark that $\W_G(A) = \varprojlim_{N} \W_{G/N}(A)$ where the inverse limit runs over open normal subgroups of $G$ and, that giving each $\W_{G/N}(A)$ the discrete topology induces a natural profinite topology on $\W_G(A)$ in which $\W_G(A)$ is complete, which follows from the proof of \cite[Thm 3.3.2]{DS}, see also \cite[pg. 357]{Elliott}. Even when $G$ is not abelian, $\W_G(A)$ is a commutative ring.
For $G = \mbf{Z}_p$, the addition and multiplication polynomials are the 
classical $p$-typical Witt addition and multiplication polynomials 
and $\mbf{W}_{\ZZ_p}(A)$ is the $p$-typical Witt vectors.

For any ring homomorphism $f \colon A \rightarrow B$ 
define $\W_G(f) \colon \W_G(A) \rightarrow \W_G(B)$ 
by applying $f$ to the coordinates: 
$$
\W_G(f)(\mbf{a}) = (f(a_T))_{T \in \mcal{F}(G)} \in \W_G(B).
$$
This is a ring homomorphism and makes $\W_G$ a covariant functor from commutative rings to commutative rings.


Packaging all the Witt polynomials together, we get a ring homomorphism 
$W : \W_G(A) \to \prod_{T \in \mcal{F}(G)} A$ which is $W_T$ in the $T$-th coordinate:
$$W(\mbf{a}) = (W_T(\mbf{a}))_{T \in \mathcal{F}(G)} = \left( \sum_{U \leq T} \varphi_T(U) a_U^{\# T/ \# U} \right)_{T \in \mcal{F}(G)}.$$
This homomorphism is called the {\it ghost map} and its coordinates 
$W_T(\mbf{a})$ are called the {\it ghost components} of $\mbf{a}$.  In some cases it 
is quite useless: if $G$ is pro-$p$ and $A$ has characteristic $p$ then 
$W(\mbf{a}) = (a_0^{\#T})_{T \in \mcal{F}(G)}$, whose dependence 
on $\mbf{a}$ only involves $a_0$.   
If $A$ fits the hypothesis of Theorem \ref{invertthm} then 
the ghost map is injective (i.e., the ghost components of $\mbf{a}$ determine $\mbf{a}$). 
Also from Theorem \ref{invertthm} the ghost map is bijective if every integer $\varphi_T(T)$ is invertible in $A$ so $\W_G(A) \cong \prod_{T \in \mcal{F}(G)} A$ by the ghost map.
That means $\W_G(A)$ is a new kind of ring only if 
some $\varphi_T(T)$ is not invertible in $A$, and especially 
if $A$ has $\varphi_T(T)$-torsion for some $T$ (e.g., $G$ is a nontrivial pro-$p$ group and $A$ has characteristic $p$). 

The coordinates on which a Witt vector is nonzero is called its {\it support}. While the ring operations in $\mbf{W}_G(A)$ are generally not componentwise, addition in $\mbf{W}_G(A)$ is 
componentwise on two Witt vectors with disjoint support. 

\begin{thm}
\label{thmRS}
Let $\{R,S\}$ be a partition of $\mcal{F}(G)$, i.e., $R \cup S = \mathcal{F}(G)$ and $R \cap S = \emptyset$. 
For every ring $A$ and any $\mbf{a} \in \mbf{W}_G(A)$, define $\mbf{r}(\mbf{a})$ and 
$\mbf{s}(\mbf{a})$ to be the Witt vectors derived from $\mbf{a}$ with 
support in $R$ and $S$: 
\begin{displaymath}
\mbf{r}(\mbf{a}) = 
\begin{cases}
a_T & \text{if $T \in R,$} \\
0 & \textrm{if $T \in S$,}
\end{cases} \quad \text{ and } \quad 
\mbf{s}(\mbf{a}) = 
\begin{cases}
0 & \text{if $T \in R,$} \\
a_T & \textrm{if $T \in S$.}
\end{cases}
\end{displaymath} Then $\mbf{a} = \mbf{r}(\mbf{a}) + \mbf{s}(\mbf{a})$ in $\mbf{W}_G(A)$. 
\end{thm}


\begin{proof} 
First we will show the result in $\mbf{W}_G(\mbf{Z}[\underline{X}])$ for the particular Witt vector 
$\mbf{x} = (X_T)_{T \in \mcal{F}(G)}$: $\mbf{x} = \mbf{r}(\mbf{x}) + \mbf{s}(\mbf{x})$ in $\mbf{W}_G(\mbf{Z}[\underline{X}])$. 
If we prove this then given any ring $A$ and $\mbf{a} \in \mbf{W}_G(A)$, 
there is a ring homomorphism 
$f \colon \ZZ[\underline{X}] \rightarrow A$ such that $f(X_T) = a_T$ for 
all $T$, and applying the ring homomorphism 
$\W_G(f) \colon \W_G(\ZZ[\underline{X}]) \rightarrow \W_G(A)$ to 
the identity 
$\mbf{x} = \mbf{r}(\mbf{x}) + \mbf{s}(\mbf{x})$
turns it into 
$\mbf{a} = \mbf{r}(\mbf{a}) + \mbf{s}(\mbf{a})$.

Since $\ZZ[\underline{X}]$ is a domain of characteristic $0$, the ghost map 
$$W \colon \W_G(\ZZ[\underline{X}]) \rightarrow \prod_{T \in \mcal{F}(G)} \ZZ[\underline{X}]$$ 
is an injective ring homomorphism, so it suffices to prove $$W(\mbf{x}) = W(\mbf{r}(\mbf{x}) + \mbf{s}(\mbf{x})).$$ 
The right side is $W(\mbf{r}(\mbf{x})) + W(\mbf{s}(\mbf{x}))$, which is a 
sum in the product ring 

\noindent $\prod_{T \in \mcal{F}(G)} \ZZ[\underline{X}]$, so its 
$T$-th coordinate for any $T$ is 
$$
W_T(\mbf{r}(\mbf{x}))  + W_T(\mbf{s}(\mbf{x})) = 
\sum_{\stackrel{U \leq T}{U \in R}} \varphi_T(U) X_U^{\# T/ \# U} + 
\sum_{\stackrel{U \leq T}{U \in S}} \varphi_T(U) X_U^{\# T/ \# U}.
$$
Since $R \cup S = \mcal{F}(G)$ and $R \cap S = \emptyset$, 
each $U$ with $U \leq T$ will lie in exactly one of $R$ or $S$, so 
$$
W_T(\mbf{r}(\mbf{x}))  + W_T(\mbf{s}(\mbf{x})) = 
\sum_{U \leq T} \varphi_T(U) X_U^{\# T/ \# U} = W_T(\mbf{x}). 
$$
Therefore
$W(\mbf{r}(\mbf{x}) + \mbf{s}(\mbf{x}))$ and 
$W(\mbf{x})$ have the same $T$-th component for all $T$, so they 
are equal, which shows 
$\mbf{r}(\mbf{x}) + \mbf{s}(\mbf{x}) = \mbf{x}$. 
\end{proof} 

One typically proves an algebraic identity in $\W_G(A)$ by reformulating it 
as an identity in a ring of Witt vectors over a polynomial ring over $\ZZ$. From now on, we will usually prove the reformulation but 
may not go through the deduction of the identity we want over $A$ from the identity proved over a 
polynomial ring; instead simply invoke functoriality.


\begin{dff}\label{TTeich}
For $a \in A$ and $T \in \mcal{F}(G)$, denote by $\omega_T(a) \in \W_G(A)$ the Witt vector with $T$-coordinate $a$ and all other coordinates $0$. We call $\omega_T(a)$ the $T$-th {\it{Teichm\"uller lift}} of $a$. 
\end{dff}

We denote the trivial $G$-set $G/G$ as $0$ and by $\omega_0(a)$ the $G/G$-th Teichm\"uller lift. The function $\omega_0 \colon A \rightarrow \W_G(A)$ generalizes the classical Teichm\"uller lift. Like the classical Teichm\"uller lift, $\omega_0$ is multiplicative. For a general formula for $\omega_T(a)\omega_{T'}(b)$, see \cite[p. 355]{Elliott}. 

An easy consequence of Theorem \ref{thmRS} is that any Witt vector $\mbf{a}$ of finite support satisfies $\mbf{a} = \sum_{U \in {\tn{Supp}}(\mbf{a})} \omega_U(a_U)$ where $\tn{Supp}(\mbf{a})$ is the support of $\mbf{a}$. 

\begin{thm}\label{multx}
For any $a \in A$ and $\mbf{b} \in \W_G(A)$, 
$$
\omega_0(a) \mbf{b} = (a^{\#T}b_T)_{T \in \mcal{F}(G)}.
$$ In particular, $ \omega_0(a)\omega_0(b) = \omega_0(ab)$.
\end{thm}

\begin{proof}
By functoriality, it suffices to show in $\W_G(\ZZ[\underline{X},\underline{Y}])$ that
$$
(X_0, 0, 0, 0, \dots)(Y_T)_{T \in \mcal{F}(G)} = (X_0^{\#T}Y_T)_{T \in \mcal{F}(G)}, 
$$
and to show this equation it 
suffices to prove the ghost components (Witt polynomial values) of both sides are equal.
Since $W_T \colon \W_G(\ZZ[\underline{X},\underline{Y}]) \rightarrow \ZZ[\underline{X},\underline{Y}]$ is multiplicative, 
\begin{eqnarray*}
W_T((X_0, 0, 0, 0, \dots)(Y_U)_{U \in \mcal{F}(G)}) & = & 
W_T(X_0, 0, 0, 0, \dots)W_T((Y_U)_{U \in \mcal{F}(G)})  \\
& = & X_0^{\#T}\sum_{U \leq T} \varphi_T(U)Y_U^{\#T/\#U} 
\end{eqnarray*}
and
\begin{eqnarray*}
W_T((X_0^{\#U}Y_U)_{U \in \mcal{F}(G)}) &=& \sum_{U \leq T} \varphi_T(U) (X_0^{\#U}Y_U)^{\#T/\#U} \\
& = & \sum_{U \leq T} \varphi_T(U) X_0^{\#T}Y_U^{\#T/\#U}. 
\end{eqnarray*}
\end{proof}

The Witt polynomial $W_T(\underline{X})$ 
becomes homogeneous of degree $\#T$ if we give $X_U$ degree $\#U$ (e.g., $X_0$ has degree $1$, not $0$).  This grading makes 
the addition and multiplication polynomials homogeneous as well: 

\begin{thm}
\label{thm:integralwitt}
Give the ring $\mbf{Z}[\underline{X},\underline{Y}]$ the grading in which the degree of $X_U$ and $Y_U$ is $\# U$. 

\begin{enumerate}
\item[$(a)$] 
For all $T$, the polynomial $S_T$ is homogeneous of degree $\# T$ and $M_T$ is homogeneous of degree $2\#T$. 

\item[$(b)$]  For all $T$, we have $S_T(\underline{X},\mbf{0}) = X_T$, 
$S_T(\mbf{0},\underline{Y}) = Y_T$, 
$M_T(\underline{X},\mbf{0}) = 0$, and $M_T(\mbf{0},\underline{Y}) = 0$.
\end{enumerate}
\end{thm}

The second part of the theorem is saying 
$S_T$ equals $X_T + Y_T$ plus monomials $X_U^iY_V^j$ where $U < T$ and $V < T$ (we cannot have $U = T$ or $V = T$ by homogeneity), while 
$M_T$ contains no monomials that are pure $\underline{X}$-terms or pure $\underline{Y}$-terms. 

\begin{proof}
(a) We will work out the homogeneity for $M_T$; the argument for $S_T$ is similar. 
Clearly $M_0(X_0,Y_0) = X_0Y_0$, which is homogeneous of degree 2 and its only monomial term contains the factors $X_0$ and $Y_0$. Let $n \geq 2$ and assume by induction for all transitive $G$-sets $U$ with $\# U < n$ that  $M_U$ is homogeneous of degree $2\#U$. Pick a transitive $G$-set $T$ with $\# T = n$. (If there are no such $G$-sets then we are vacuously done.) Solving for $M_T$ in (\ref{MTF}) in $\QQ[\underline{X},\underline{Y}]$,  
\begin{eqnarray*}
M_T & = & 
\frac{1}{\varphi_T(T)} \left( \sum_{U_1 \leq T} \varphi_T(U_1)X_{U_1}^{ \frac{\# T}{\# U_1}}\sum_{U_2 \leq T}  \varphi_T(U_2)Y_{U_2}^{\frac{\# T}{\# U_2}} - \sum_{U < T} \varphi_T(U) M_U^{\frac{\# T }{\# U}} \right) \\
& = & \frac{1}{\varphi_T(T)} \left( \sum_{U_1,U_2 \leq T} \varphi_T(U_1)\varphi_T(U_2) X_{U_1}^{\frac{\# T}{\# U_1}} Y_{U_2}^{\frac{\# T}{\# U_2}} - \sum_{U < T} \varphi_T(U) M_U^{\frac{\# T}{\# U}} \right).
\end{eqnarray*} 
By the inductive hypothesis, 
each $M_U$ for $U < T$ is homogeneous of degree $2\#U$, so $M_U^{\#T/\#U}$ is homogeneous of 
degree $2\#T$. 
By the definition of the grading, $X_{U_1}^{\#T/\#U_1}$ has degree $\#T$ and 
$Y_{U_2}^{\#T/\#U_2}$ has degree $\#T$, so 
$X_{U_1}^{\#T/\# U_1} Y_{U_2}^{\# T / \# U_2}$ has degree  $2\#T$.

(b) We will work out the result for multiplication polynomials. Since $M_0(\underline{X},\mbf{0})$ is $X_00 = 0$, we may assume $T \neq 0$.  Set every $Y_U$ to $0$ in the recursive formula for $M_T$ above. Then the formula 
tells us $M_T(\underline{X},\mbf{0})$ is equal to 
$$
\frac{1}{\varphi_T(T)} \left( \sum_{U_1,U_2 \leq T} \varphi_T(U_1)\varphi_T(U_2) X_{U_1}^{\#T/\# U_1}\cdot 0 - \sum_{U < T} \varphi_T(U) M_U(\underline{X},\mbf{0})^{\# T / \#U } \right). 
$$
Using the inductive hypothesis, each term is 0. 
\end{proof}

So the polynomials $S_T( \{X_U^{\# U}, Y_U^{\# U}\}_{U \in \mcal{F}(G)})$ are genuine homogeneous polynomials of degree $\# T$ (replace $X_U$ and $Y_U$ with
their $\# U$-th powers everywhere), and similarly for the multiplication polynomials.

\subsection{Units}

The present goal is to demonstrate that $\W_G(k)$ is a local ring when $G$ is pro-$p$ and $k$ is a field of characteristic $p$. We state a couple of needed lemmas about divisibility relationships among $\varphi_T(U)$ for $U, T \in \mcal{F}(G)$ whose proofs are straightforward and left to the reader.

\begin{lem}
\label{lem:propdivis}
If $G$ is a pro-$p$ group and $T \geq U$ in $\mcal{F}(G)$ with $U$ nontrivial, then $\varphi_T(U) \equiv 0 \bmod p$.
\end{lem}

\begin{lem}
\label{lem:cong}
Let $G$ be a nontrivial pro-$p$ group, $U < T$ in $\mcal{F}(G)$ and $s \in p\ZZ.$ Then 
$$
\frac{\varphi_T(U)}{\varphi_T(T)} s^{\# T / \# U } \in p\ZZ.
$$
\end{lem}

We apply these to study a particularly important family of ideals in $\W_G(A)$. 

\begin{dff}\label{Indef}
For $n \in \mbf{Z}^+$, set $$I_n(G,A) = \{ \mbf{a} \in \W_G(A) : a_T = 0 \mbox{ for }  \# T < n \}.$$
\end{dff}

These are the Witt vectors $\mbf{a}$ with support in $\{ T : \# T \geq n \}$. 

\begin{lem}
\label{lem:fil1}
Each $I_n(G,A)$ is an ideal in $\mbf{W}_G(A)$. 
\end{lem}
\begin{proof}
Each addition polynomial $S_T$ depends only on variables indexed by (isomorphism classes of) finite transitive $G$-sets $U \leq T$ and 
has no constant term by Theorem \ref{thm:integralwitt}. So if two Witt vectors are in $I_n(G,A)$ then their sum is also in $I_n(G,A)$. 

It remains to show for any $\mbf{a} \in \W_G(A)$ and $\mbf{b} \in I_n(G,A)$ that $\mbf{a}\mbf{b} \in I_n(G,A)$. 
Set $\mbf{c} = \mbf{a}\mbf{b}$.
By the definition of multiplication in $\W_G(A)$, $c_T = M_T(\mbf{a},\mbf{b})$ for any $T$. 
If $\# T < n$, $b_U = 0$ for $U \leq T$ by hypothesis.  Since 
$M_T(\underline{X},\underline{Y})$ only depends on $Y_U$ for $U \leq T$, 
$c_T = M_T(\mbf{a},\mbf{b}) = M_T(\mbf{a},\mbf{0})$ and 
$M_T(\mbf{a},\mbf{0}) = 0$ by 
Theorem \ref{thm:integralwitt}.
\end{proof}

\begin{lem}
\label{lem:weakIFill}
Let $G$ be a pro-$p$ group. For all rings $A$ of characteristic $p$ and nonnegative integers $n$, $$I_{p}(G,A) I_{p^n}(G,A) \subset I_{p^{n+1}}(G,A)$$ in $\W_G(A)$. 
\end{lem}
\begin{proof}
If $n = 0$ the result is clear since $I_1(G,A) = \W_G(A)$, so without loss of generality assume $n \geq 1$. 

We will derive a mod $p$ congruence for particular Witt vectors over the ring $R = \mbf{Z}[\underline{X},\underline{Y}]$, which will be sufficient using functoriality. 
Define $\mbf{x}$ and $\mbf{y}$ in $\W_G(R)$ by
\begin{displaymath}
x_T = 
\begin{cases}
0, & \text{if $ T = 0$,} \\
X_T, & \textrm{if $T \neq 0$,}
\end{cases} \quad \text{ and } \quad 
y_T = 
\begin{cases}
0, & \text{if $\# T < p^n$,} \\
Y_T, & \textrm{otherwise.}
\end{cases}
\end{displaymath} 
Set $\mbf{z} = \mbf{xy}$. Our aim is to show 
\begin{equation}
\# T < p^{n+1} \Longrightarrow z_T \equiv 0 \bmod pR. 
\end{equation}
We argue by induction on $\#T$. Clearly $z_0 = x_0y_0 = 0$. 
Let $\# T = p^r < p^{n+1}$ with $r \geq 1$ and 
assume by induction that for all $\# U < p^r$, $z_U \equiv 0 \bmod pR$. 
Since the $T$-th Witt polynomial is a multiplicative function  
$W_T \colon \W_G(R) \rightarrow R$, 
$W_T(\mbf{z}) = W_T(\mbf{x})W_T(\mbf{y})$: 
$$\sum_{U \leq T} \varphi_T(U) z_U^{\# T / \#U } = \sum_{T_1,T_2 \leq T} \varphi_T(T_1)\varphi_T(T_2) x_{T_1}^{\#T / \# T_1} y_{T_2}^{\# T / \# T_2}.$$ Solving this equation for $z_T$ in $\QQ[\underline{X},\underline{Y}]$, 
\begin{equation}\label{weakztf}
z_T = \sum_{T_1,T_2 \leq T} \frac{\varphi_T(T_1)\varphi_T(T_2)}{\varphi_T(T)} x_{T_1}^{\#T / \# T_1} y_{T_2}^{\# T / \# T_2} - \sum\limits_{U < T} \frac{\varphi_T(U)}{\varphi_T(T)} z_U^{\#T / \#U}.
\end{equation}

Since $z_U \in pR$ for $U <T$, 
the second term in (\ref{weakztf}) is $0$ mod $pR$ by Lemma \ref{lem:cong}.
In the first term in (\ref{weakztf}), 
we can assume $T_1 \neq 0$ since $x_0 = 0$. 
If $\# T_2 < p^n$ then $y_{T_2} = 0$, 
so we only need to consider $T_2$ where $\# T_2 \geq p^n$. 
Since $\# T < p^{n+1}$, $T_2 = T$. 
Then (\ref{weakztf}) becomes $$ z_T = \sum_{0 < T_1 \leq T} \varphi_T(T_1) x_{T_1}^{\#T / \# T_1} y_{T} - \sum\limits_{U < T} \frac{\varphi_T(U)}{\varphi_T(T)} z_U^{\#T / \#U}.$$ 
Since $\varphi_T(T_1)$ is an integral multiple of $p$ by Lemma \ref{lem:propdivis}, $z_T \equiv 0 \bmod pR$.     
\end{proof} An immediate useful corollary follows. 

\begin{cor}
\label{cor:IFillpower}
For any pro-$p$ group $G$ and ring $A$ of characteristic $p$, 

\noindent $I_p(G,A)^m \subset I_{p^m}(G,A)$. 

\end{cor}
\begin{proof}
This follows directly from repeated applications of Lemma \ref{lem:weakIFill}. 
\end{proof}
We are now set to prove the main result of this section. 

\begin{thm}
\label{units}
Let $G$ be a pro-$p$ group and $A$ be a ring of 
characteristic $p$. The units in $\W_G(A)$ are $\W_G(A)^\times = \{ \mbf{a} : a_0 \in A^\times \}$. Consequently, when $(A,\goth{n})$ is local, $\W_G(A)$ is a local ring with maximal ideal $\goth{m} = \{ \mbf{a} \in \mbf{W}_G(A) : a_0 \in \goth{n} \}$. 
\end{thm}
\begin{proof}
Obviously $\W_G(A)^\times \subset \{ \mbf{a} : a_0 \in A^\times \}$. To prove the reverse 
inclusion, let $\mbf{a} \in \W_G(A)$ have $a_0 \in A^\times$.  
By Theorem \ref{multx} the Witt vector $(a_0,0,0,\dots)$ is a unit, with inverse $(a_0^{-1},0,0,\dots)$,
so it suffices to show $(a_0^{-1},0,0,\dots)\mbf{a}$ is a unit.  The first coordinate 
of this product is 1, so we are reduced to showing a Witt vector with first coordinate 1 is 
a unit.   That is, we can assume $a_0 = 1$.  By Theorem \ref{thmRS}, 
$$
\mbf{a} = (1,0,0,\dots) + (0,\{a_T\}_{T \not= 0}) \in 1 + I_p(G,A). 
$$
Since $I_p(G,A)^m \subset I_{p^m}(G,A)$ by Corollary \ref{cor:IFillpower}, 
we can invert ${\mbf a}$ 
using a geometric series
since $\W_G(A)$ is complete in the profinite topology. 
\end{proof}

\section{The Frame of $\ZZ_p^d$}\label{secd2}

The subgroup structure of $\ZZ_p^d$ is homogeneous in the sense that every open subgroup is isomorphic to $\ZZ_p^d$ (although there is not a canonical choice of isomorphism, unlike the case when $d = 1$).  
A subgroup is open if and only if it has finite index. 
If $H$ is an open subgroup 
there is a $\ZZ_p$-basis $\{e_1,e_2,\ldots,e_d\}$ for $G$ such that $G = \ZZ_p e_1 \oplus \ZZ_p e_2 \oplus \cdots \oplus \ZZ_p e_d$ and $H = \ZZ_p p^{a_1} e_1 \oplus \ZZ_p p^{a_2} e_2 \oplus \cdots \oplus \ZZ_p p^{a_d} e_d$ for some $a_1,\ldots,a_d\geq 0$. 
The $G$-set $G/H$ has the form 
\begin{equation}\label{GHa}
G/H  \cong \ZZ_p/p^{a_1}\ZZ_p \times \ZZ_p/p^{a_2}\ZZ_p \times \cdots \times 
\ZZ_p/p^{a_d}\ZZ_p.
\end{equation}

As a group, $G/H$ is usually a product of $d$ nontrivial cyclic $p$-groups.  
For some $H$, $a_i = 0$ for all but one $i$, making $T = G/H$ a cyclic group. 

\begin{dff}
\label{dff:cyclic}
Let $G$ be any profinite group and $N$ be an open normal subgroup. A $G$-set $T \cong G/N$ where $G/N$ is a cyclic group is called a {\it{cyclic}} $G$-set. 
\end{dff} 

An important property about cyclic $G$-sets $T$ when $G$ is a pro-$p$ group is that $\{ U \in \mcal{F}(G) \colon U < T\}$, which is called the {\it{strict downset of $T$}}, with its induced order is a chain. 

\

{\it Throughout the rest of this section $G = \ZZ_p^d$ with $d \geq 2$}.

\

From the large number of $\ZZ_p$-bases of $G$ it is reasonable to expect that there are many cyclic $G$-sets of each size as the size grows. We will use cyclic $G$-sets later (Lemma \ref{lem:nicyclic}) to find nonisomorphic $G$-sets of the same size with the same strict downsets, 
i.e., the $G$-sets lying below them in the frame of $G$ are the same, which will be important in Section \ref{sec:Zpdnotnoth}.

For each $T \in \mcal{F}(G)$, there is an open subgroup $H \subset G$ with $T \cong G/H$, and $H$ is uniquely determined by $T$ since $G$ is abelian: $H$ is the common stabilizer of all points in $T$.  
Using a $\ZZ_p$-basis $\{e_1,e_2,\ldots,e_d\}$ of $G$ we can
identify $H$ with $p^{a_1}\ZZ_p \times p^{a_2} \ZZ_p \times \cdots \times p^{a_d} \ZZ_p$ (this amounts to applying an automorphism of $G$), so the $G$-sets below $G/H$ in $\mcal{F}(G)$ are in one-to-one correspondence with the (open) subgroups of $G$ that contain 
$p^{a_1}\ZZ_p \times p^{a_2} \ZZ_p \times \cdots \times p^{a_d} \ZZ_p$.  

The frame of $\mcal{F}(\ZZ_p)$ is linearly ordered so every element has a unique cover ($G$-set lying directly above it). Since every subgroup of $\ZZ_p^d$ is isomorphic to $\ZZ_p^d$ and there are $p^{d-1} + \ldots +p +1$ maximal subgroups of $\ZZ_p^d$ the number of covers of each element of $\mcal{F}(G)$ is the same. 

\begin{thm}
\label{thm:coversinZpd}
Let $T \in \mcal{F}(\ZZ_p^d)$. It has  $p^{d-1} + \ldots + p +1$ covers in $\mcal{F}(\ZZ_p^d)$ and in the case $T \geq \ZZ_p^d/p\ZZ_p^d$, it covers $p^{d-1} + \ldots + p + 1$ elements of $\mcal{F}(\ZZ_p^d)$.
\end{thm}
\begin{proof}
Let $H$ be the stabilizer of $T$. Covers of $T$ correspond to subgroups of $H$ with index $p$, which correspond to maximal subgroups of $H/pH \cong (\ZZ/p\ZZ)^d.$ By duality the number of such subgroups is the number of subgroups of $H/pH$ with size $p$, namely the number of linear subspaces of $(\ZZ/p\ZZ)^d$. That is $(p^d - 1)/(p-1) = 1 + p + \ldots + p^{d-1}$. 

Similarly, if $T$ covers $U$ then $U = \ZZ_p^d/K$ where $H \subset K$ and $[K : H] = p$, so the sets which $T$ covers correspond to subgroups of size $p$ in $\ZZ_p^d/H$. The number of such subgroups is the same as the number of subgroups of index $p$. When $\ZZ_p^d/H \geq \ZZ_p^d/p\ZZ_p^d$, we have $H \subset p \ZZ_p^d$. All subgroups of $\ZZ_p^d$ with index $p$ contain $p\ZZ_p^d$, so the number of subgroups of index $p$ in $\ZZ_p^d/H$ and $\ZZ_p^d/p\ZZ_p^d$ is the same. 

\end{proof}

In some calculations, it was noticed that many pairs of coordinates in sums or products of certain Witt vectors are equal. This is described formally by Lemmas \ref{lem:nicyclicsum} and \ref{lem:nicyclicprod} below and motivates the following definition.

\begin{dff}
\label{dff:linked}
A nonisomorphic pair of $G$-sets $T$ and $T'$, whose strict downsets agree, is called {\it{linked}}.
\end{dff}

For a linked pair of $G$-sets $T$ and $T'$,
$\# T = \#T'$ since for any $U < T$ of maximal size, $\#T = p\#U$. 
Therefore $\#T' = p\#U = \#T$. 
An example of such a pair $T$ and $T'$ is labeled in Figure \ref{fig:framelinked}. The pair $V_1$ and $V_2$ in Figure \ref{fig:framelinked} is not linked since the $G$-set $U$ is below $V_2$ and not below $V_1$. 

\begin{figure}[htb]
\begin{center}
\scalebox{0.8}{
\begin{tikzpicture}[smooth]


\fill[black] (0,0) circle (0.1cm);

\fill[black] (1,0) circle (0.1cm);
\fill[black] (1,-1) circle (0.1cm);
\fill[black] (1,1) circle (0.1cm);

\fill[black] (3,1.25) circle (0.1cm);
\fill[black] (3,0.75) circle (0.1cm);

\fill[black] (3,0.25) circle (0.1cm);
\fill[black] (3,-0.25) circle (0.1cm);

\fill[black] (3,-0.75) circle (0.1cm);
\fill[black] (3,-1.25) circle (0.1cm);

\fill[black] (3,2.5) circle (0.1cm);

\fill[black] (5,3) circle (0.1cm);
\fill[black] (5,2.5) circle (0.1cm);
\fill[black] (5,2) circle (0.1cm);

\fill[black] (5,-1.2) circle (0.1cm);
\fill[black] (5,-1.45) circle (0.1cm);

\fill[black] (5,-0.65) circle (0.1cm);
\fill[black] (5,-0.9) circle (0.1cm);

\fill[black] (5,-0.15) circle (0.1cm);
\fill[black] (5,-0.4) circle (0.1cm);

\fill[black] (5,0.35) circle (0.1cm);
\fill[black] (5,0.1) circle (0.1cm);

\fill[black] (5,0.85) circle (0.1cm);
\fill[black] (5,0.6) circle (0.1cm);

\fill[black] (5,1.35) circle (0.1cm);
\fill[black] (5,1.1) circle (0.1cm);

\fill[black] (7,0) circle (0.05cm);
\fill[black] (7.2,0) circle (0.05cm);
\fill[black] (7.4,0) circle (0.05cm);

\fill[black] (5.5,2.5) circle (0.05cm);
\fill[black] (5.7,2.5) circle (0.05cm);
\fill[black] (5.9,2.5) circle (0.05cm);

\fill[black] (4,3.25) circle (0.05cm);
\fill[black] (4.2,3.45) circle (0.05cm);
\fill[black] (4.4,3.65) circle (0.05cm);

\draw[-] (0,0) -- (1,0);
\draw[-] (0,0) -- (1,-1);
\draw[-] (0,0) -- (1,1);

\draw[-] (1,1) -- (3,1.25);
\draw[-] (1,1) -- (3,0.75);

\draw[-] (1,0) -- (3,0.25);
\draw[-] (1,0) -- (3,-0.25);

\draw[-] (1,-1) -- (3,-0.75);
\draw[-] (1,-1) -- (3,-1.25);

\draw[-] (-1,1.5) -- (6,1.5);

\draw[-] (1,1) -- (3,2.5);
\draw[-] (1,0) -- (3,2.5);
\draw[-] (1,-1) -- (3,2.5);

\draw[-] (3,2.5) -- (5,3);
\draw[-] (3,2.5) -- (5,2.5);
\draw[-] (3,2.5) -- (5,2);

\draw[-] (3,1.25) -- (5,1.35);
\draw[-] (3,1.25) -- (5,1.1);

\draw[-] (3,0.75) -- (5,0.85);
\draw[-] (3,0.75) -- (5,0.6);

\draw[-] (3,0.25) -- (5,0.35);
\draw[-] (3,0.25) -- (5,0.1);

\draw[-] (3,-0.25) -- (5,-0.15);
\draw[-] (3,-0.25) -- (5,-0.4);

\draw[-] (3,-0.75) -- (5,-0.65);
\draw[-] (3,-0.75) -- (5,-0.9);

\draw[-] (3,-1.25) -- (5,-1.2);
\draw[-] (3,-1.25) -- (5,-1.45);

\draw[-] (3,1.25) -- (5,3);
\draw[-] (3,0.75) -- (5,3);

\draw[-] (3,0.25) -- (5,2.5);
\draw[-] (3,-0.25) -- (5,2.5);

\draw[-] (3,-0.75) -- (5,2);
\draw[-] (3,-1.25) -- (5,2);

\draw (5,3) circle (0.15cm);
\draw (5,2.5) circle (0.15cm);

\draw (5,-1.2) circle (0.15cm);
\draw (5,-1.45) circle (0.15cm);

\draw (3,-0.25) circle (0.15cm);

\node at (2.6, -0.35) {$U$} {};

\node at (5.4,-1.15) {$T_1$} {};
\node at (5.4,-1.5) {$T_2$} {};

\node at (5.4,3) {$V_1$} {};
\node at (5.4,2.2) {$V_2$} {};

\end{tikzpicture}
}
\end{center}
\caption{Linked and non-linked $\ZZ_2^2$-sets.}
\label{fig:framelinked}
\end{figure}


Our first task is to show that linked pairs of $G$-sets of arbitrarily large size exist for $G = \ZZ_p^d$ for $d \geq 2$. Of course, for $d = 1$, there are no linked $\ZZ_p$-sets. 

\begin{lem}
\label{lem:nicyclic}
For each $n \geq 1$ and cyclic $G$-set $T$ of size $p^{n-1}$, there is a linked pair of cyclic $G$-sets $T_1$ and $T_2$ covering $T$. 
\end{lem}

\begin{proof}
Since $G$ is abelian, $G$-sets $G/H$ for different (open) subgroups $H$ are nonisomorphic. 
Since $d \geq 2$, there is more than one subgroup of index $p$. 
That settles the case $n = 1$. 

Now let $n \geq 2$ and write $T = G/H$. There is a unique chain of subgroups $$ H = K_0 \subset K_1 \subset \cdots \subset K_{n-1} = G $$
where $[K_i:H] = p^i$ for all $i$.

There is a $\ZZ_p$-basis of $\ZZ_p^d$, $\{e_1,\ldots,e_d\}$, such that $H = \bigoplus_{i=1}^{d-1} \ZZ_p e_i \oplus \ZZ_p p^{n-1} e_d$.
Consider the subgroups $H_1 = \bigoplus_{i=1}^{d-1} \ZZ_p e_i \oplus \ZZ_p p^{n}e_d$ and 
$H_2 = \bigoplus_{i=1}^{d-2} \ZZ_p e_i \oplus \ZZ_p pe_{d-1} \oplus \ZZ_p(e_{d-1}+ p^{n-1}e_{d})$. Clearly $H_1 \subset H$ and $H_2 \subset H$, and both $G/H_1$ and $G/H_2$ are cyclic. Since $\# G/H_1$ and $\# G/H_2$ are both $p^{n}$, $G/H_1$ and $G/H_2$ are covers of $G/H$.  

Set $T_1 = G/H_1$ and $T_2 = G/H_2$. Since $H_1 \not= H_2$, $T_1$ and $T_2$ are nonisomorphic. To show their strict downsets agree, we use the cyclic condition. The $G$-sets $U < T_1$ are $G/K_i$ for $i = 0,\dots,n-1$.  The same argument applies to $T_2$, so the same $G$-sets lie strictly below $T_1$ and $T_2$ in the frame of $G$.
\end{proof}

\section{$\W_{\ZZ_p^d}(k)$ is not Noetherian for $d \geq 2$}\label{sec:Zpdnotnoth}

When $k$ is a perfect field of characteristic $p$ and $d \geq 2$, 
one might expect the rings $\W_{\ZZ_p^d}(k)$ generalize 
the classical Witt vectors $\W_{\ZZ_p}(k)$ in the same way that 
$k[[X_1,\dots,X_d]]$ generalizes $k[[X_1]]$: 
the power series ring in $d$ variables over a field is a complete local Noetherian domain with dimension $d$. 

It essentially follows from \cite{DS} that $\W_{\ZZ_p^d}(k)$ is not a domain but is a local ring whether or not $k$ is perfect. Since it is also complete in its profinite topology 
it is plausible to guess, by analogy to $k[[X_1,\dots,X_d]]$, 
that the maximal ideal of $\W_{\ZZ_p^d}(k)$ 
is generated by the Witt vectors $\omega_T(1)$ for $\#T = p$, but 
we will see this is false in a very strong way: $\W_{\ZZ_p^d}(k)$ is not Noetherian, 
whether or not $k$ is perfect. This is because, as we will see, 
the square of the maximal ideal of $\W_{\ZZ_p^d}(k)$ is much smaller than intuition suggests. 


To prove $\W_{\ZZ_p^d}(k)$ is not Noetherian when $d \geq 2$, note that Lemma \ref{lem:nicyclic} guarantees lots of linked paris of $\ZZ_p^d$-sets and we next prove two lemmas that describe the behavior of sums and products on these linked pairs of coordinates.. The upshot, Theorem \ref{thm:equalc}, is that 
Witt vectors in the square of the maximal ideal of $\W_{\ZZ_p^d}(k)$ have built-in redundancies in linked coordinates which occur arbitrarily far out into the frame of $\ZZ_p^d$ when $d \geq 2$. 

We continue to use the convention that $G$ stands for $\ZZ_p^d$.

\begin{lem}
\label{lem:nicyclicsum}
Let $A$ be a ring. Given linked $G$-sets $T$ and $T'$ and a finite collection of Witt vectors $\mbf{a}_1,\mbf{a}_2,\ldots,\mbf{a}_m \in \W_G(A)$ such that $(\mbf{a}_i)_T = (\mbf{a}_i)_{T'}$ for $1 \leq i \leq m$, set $\mbf{s} = \sum\limits_{i=1}^m \mbf{a}_i$. Then $s_T = s_{T'}$. 
\end{lem}

\begin{proof}
By induction it suffices to check the case $m=2$.  

Since $T$ and $T'$ are linked i.e., the $G$-sets strictly below $T$ and $T'$ in $\mcal{F}(G)$ are the same, $\# T = \# T'$ 
and the Witt polynomials $W_T(\underline{X})$ and $W_{T'}(\underline{X})$ are the same 
up to the roles of $X_T$, $Y_T$ and $X_{T'},Y_{T'}$ in them.  
Therefore the two sum polynomials $S_T$ and $S_{T'}$ are the same up to the roles of $X_T$, $Y_T$, $X_{T'}$, and $Y_{T'}$ in them 
(and we know how $X_T$, $Y_T$, $X_{T'}$ and $Y_{T'}$ appear in $S_T$ and $S_{T'}$ by Theorem \ref{thm:integralwitt}).  
So when Witt vectors $\mbf{a}_1$ and $\mbf{a}_2$ 
have the same $T$ and $T'$ coordinates, since $T$ and $T'$ are linked we have $S_T(\mbf{a}_1,\mbf{a}_2) = S_{T'}(\mbf{a}_1,\mbf{a}_2)$. Thus $(\mbf{a}_1 + \mbf{a}_2)_T = (\mbf{a}_1 + \mbf{a}_2)_{T'}$. \end{proof}

Unlike the previous lemma, the next one is specific to rings of characteristic $p$.

\begin{lem}
\label{lem:nicyclicprod}
Let $A$ be a ring of characteristic $p$. Given a pair of linked $G$-sets $T$ and $T'$ and elements $\mbf{a}$ and $\mbf{b}$ of $\W_G(A)$ such that
$a_0 = b_0 = 0$, the product $\mbf{m} = \mbf{ab}$ satisfies $m_T = m_{T'}$.
\end{lem}

\begin{proof}
Let $R =\mbf{Z}[\underline{X},\underline{Y}]$ and define $\mbf{x}$ and $\mbf{y}$ in $\W_{G}(R)$ by $x_U = X_U$ and $y_U = Y_U$ for all $U$.
Set $\mbf{z} = \mbf{x}\mbf{y}$. We will show 
$$ z_T - z_{T'} \equiv (x_T - x_{T'})y_0^{p^n} + (y_T - y_{T'})x_0^{p^n} \bmod pR.$$
Since $A$ has characteristic $p$, it would then follow 
by functoriality that $\mbf{ab}$ has equal $T$ and $T'$ coordinates in $\W_{G}(A)$. 

Since the $T$-th Witt polynomial is a multiplicative function 
$W_T \colon \W_{G}(R) \rightarrow R$, 
$W_T(\mbf{z}) = W_T(\mbf{x})W_T(\mbf{y})$: 
$$
\sum\limits_{U \leq T} \# U z_U^{\#T/\#U} = \left(\sum_{U \leq T} \#U x_U^{\#T/\#U}\right)\left(\sum_{U \leq T} \#U y_U^{\#T/\#U}\right).
$$ 
Isolating the $z_T$ term, 
\begin{equation}
\label{eq:nicyclicprod1}
\# T z_{T} = \left(\sum_{U \leq T} \#U x_U^{\#T/\#U}\right)\left(\sum_{U \leq T} \#U y_U^{\#T/\#U}\right) - \sum\limits_{U < T} \# U z_U^{\#T/\#U}.
\end{equation}
Likewise looking at the $T'$ coordinate we have 
\begin{equation}
\label{eq:nicyclicprod2}
\# T' z_{T'} = \left(\sum_{U \leq T'} \#U x_U^{\# T'/\#U}\right)\left(\sum_{U \leq T'} \#U y_U^{\# T'/\#U}\right) - \sum\limits_{U < T'} \# U z_U^{\# T'/\#U}.
\end{equation}
Since $T$ and $T'$ are linked, $\# T = \# T'$ and $\{U < T\} = \{U < T'\}$ so the $z$-terms being subtracted on the right side of (\ref{eq:nicyclicprod1}) and (\ref{eq:nicyclicprod2}) are the same. 
Subtracting (\ref{eq:nicyclicprod2}) from (\ref{eq:nicyclicprod1}) and setting 
$\#T = \#T' = p^n$, we have
\begin{eqnarray*}
p^n z_{T} - p^n z_{T'} & = & \left(\sum_{U \leq T} \#U x_U^{\frac{p^n}{\#U}}\right)\left(\sum_{U \leq T} \#U y_U^{\frac{p^n}{\# U}}\right) - \sum\limits_{U < T} \# U z_U^{\frac{p^n}{\# U}} \\
 & & - \left(\sum_{U \leq T'} \#U x_U^{\frac{p^n}{\#U}}\right)\left(\sum_{U \leq T'} \#U y_U^{\frac{p^n}{\# U}}\right) + \sum\limits_{U < T'} \# U z_U^{\frac{p^n}{\# U}}\\
 & = & p^{2n}x_Ty_T +  p^nx_T \sum_{U < T} \#U y_U^{\frac{p^n}{\# U}} + p^n y_T \sum_{U < T} \#U x_U^{\frac{p^n}{\# U}} \\ 
 & & - p^{2n}x_{T'}y_{T'} - p^n x_{T'} \sum_{U < T'} \#U y_U^{\frac{p^n}{\# U}} - p^n y_{T'} \sum_{U < T'} \#U x_U^{\frac{p^n}{\#U}} \\
 & \equiv & p^n x_T y_0^{p^n}  + p^n x_0^{p^n} y_T  - p^n x_{T'} y_0^{p^n}  - p^n x_0^{p^n} y_{T'} \bmod p^{n+1}R \\
 & \equiv & p^n(x_Ty_0^{p^n}+ x_0^{p^n}y_T - x_{T'}y_0^{p^n} - x_0^{p^n}y_{T'}) \bmod p^{n+1}R \\
 & \equiv & p^n( (x_T - x_{T'})y_0^{p^n} +(y_T - y_{T'})x_0^{p^n}) \bmod p^{n+1}R.\\
\end{eqnarray*} 
So $z_T - z_{T'} \equiv (x_T - x_{T'})y_0^{p^n} +(y_T - y_{T'})x_0^{p^n} \bmod pR$. 
 \end{proof}
 
\begin{remark}
The characteristic $p$ hypothesis in Lemma \ref{lem:nicyclicprod} is necessary. Consider the case $G=\ZZ_p^d$ for $d \geq 2$ and $\W_G(\ZZ)$. Any pair of nonisomorphic $G$-sets of size $p$ is linked. Given two vectors $\mbf{a},\mbf{b} \in \W_G(\ZZ)$ such that $a_0 = b_0 = 0$, set $\mbf{m} = \mbf{ab}$. For $T \in \mcal{F}(G)$ such that $\# T = p$, the formula for $M_T$ in Example \ref{xmp:MTST} implies $m_T = pa_Tb_T$, which depends on $T$ for suitable choices of $\mbf{a}$ and $\mbf{b}$. 
\end{remark}

The point of these last two lemmas is that for linked $G$-sets $T$ and $T'$, the 
$T$ and $T'$ coordinates of a sum are the same if we make an assumption about the 
$T$ and $T'$ coordinates of the summands, while 
the $T$ and $T'$ coordinates of a product are the same if we make an assumption about the 
coordinates of the factors at the trivial $G$-set.

Let $k$ be a field of characteristic $p$.  For $n \geq 0$, recall the notation
$$I_{p^n} = I_{p^n}(G,k) = I_{p^n}(\ZZ_p^d,k) = \{ \mbf{a} \in \W_{\ZZ_p^d}(k) ; a_T = 0 \text{ if } \# T < p^n\}.$$
The unique maximal ideal of $\W_G(k)$ is $\goth{m} = I_p(G,k)$. 

\begin{thm}
\label{thm:equalc}
Let $T$ and $T'$ be linked $G$-sets. Any element of $\goth{m}^2$ has equal $T$ and $T'$ coordinates. 
\end{thm}

\begin{proof}
Any element of $\goth{m}^2$ is $\mbf{a}_1\mbf{b}_1 + \cdots + \mbf{a}_r\mbf{b}_r$ for some 
$\mbf{a}_i$ and $\mbf{b}_i$ in $\goth{m}$. 
Lemma \ref{lem:nicyclicprod} shows that $\mbf{a}_i\mbf{b}_i$ has the same $T$ and $T'$ coordinates for all $i$.
Applying Lemma \ref{lem:nicyclicsum} shows that this equality of coordinates is preserved when passing to the sum of these products over all $i$. 
\end{proof} 

Since Lemma \ref{lem:nicyclic} shows that there are linked $G$-sets of arbitrarily large size, 
Theorem \ref{thm:equalc} puts infinitely many constraints on elements of $\goth{m}^2$. 
Theorem \ref{thm:equalc} was first observed in examples, where many 
redundancies were noticed in different coordinates of a product of elements of 
$\W_{\ZZ_2^2}(\mbf{F}_2[\underline{X},\underline{Y}])$ that each have first coordinate 0.  This is also how the importance of linked $G$-sets was discovered. We are now set to use them to prove our first main structural theorem.


\begin{thm}
\label{thm:Zpsqnotnoth}
The maximal ideal of $\W_G(k)$ is not finitely generated, so $\W_{G}(k)$ is not Noetherian.
\end{thm}

\begin{proof}
Using Lemma \ref{lem:nicyclic}, for each $n \geq 1$ there are linked $G$-sets $T_n$ and $T_n'$ in $\mcal{F}(G)$ of size $p^n$. In particular, $\{ U : U < T_n \} = \{ U : U < T_n' \}$. 
The Witt vector $\omega_{T_n}(1)$ lies in $\goth{m}$. 
For each $r \geq 1$ we will show 
the $r$ Witt vectors $\omega_{T_1}(1), \dots, \omega_{T_r}(1)$ 
are linearly independent over $k$ in $\goth{m}/\goth{m}^2$. 

In $\goth{m}/\goth{m}^2$, suppose we have a $k$-linear relation 
$$
\alpha_1\omega_{T_1}(1) + \cdots + \alpha_r\omega_{T_r}(1) \equiv 0 \bmod \goth{m}^2,
$$
for some $\alpha_i \in k$.  The product $\alpha_i\omega_{T_i}(1)$ in 
$\goth{m}/\goth{m}^2$ really means, as a Witt vector, 
$\omega_0(\alpha_i)\omega_{T_i}(1) \bmod \goth{m}^2$. 
By Theorem \ref{multx}, $\omega_0(\alpha_i)\omega_{T_i}(1) = \omega_{T_i}(\alpha_i^{p^i})$.
Therefore
\begin{equation}\label{m2s}
\omega_{T_1}(\alpha_1^p) + \cdots + \omega_{T_r}(\alpha_r^{p^r}) \equiv 0 \bmod \goth{m}^2.
\end{equation}
Since the supports of $\omega_{T_i}(\alpha_i^{p^i})$ for $1 \leq i \leq r$ are disjoint, by Theorem \ref{thmRS} these Witt vectors can be added coordinatewise:  the left side of (\ref{m2s})
is the Witt vector with $T_i$-coordinate $\alpha_i^{p^i}$ and other coordinates equal to 0.  
Since this sum is in $\goth{m}^2$,  
its $T_i$- and $T_i'$-coordinates are the same by Theorem \ref{thm:equalc}. 
The $T_i'$-coordinate is 0 for all $i$, so $\alpha_i^{p^i} = 0$ for all $i$. 
Thus every $\alpha_i$ is 0 in $k$. 

Since we have found $r$ linearly independent elements of 
$\goth{m}/\goth{m}^2$ for any $r \geq 1$, its $k$-dimension is infinite. Therefore $\goth{m}$ is not finitely generated. 
\end{proof}

%

\begin{remark}
For an infinite pro-$p$ group $G$ with arbitrarily large pairs of linked normal $G$-sets, the results of this section go through. For such pro-$p$ groups, $\W_G(k)$ is not Noetherian when $k$ is a field (or ring) of characteristic $p$. 
\end{remark}

\section{Reducedness of $\W_G(k)$}\label{sec:reduced}

Although $\W_{\ZZ_p^d}(k)$ with $d > 2$ is not a domain and is not Noetherian, 
this is not a pathological state of affairs among $p$-adic rings. 
For instance, the ring $C(\ZZ_p,\QQ_p)$ of continuous functions 
from $\ZZ_p$ to $\QQ_p$ is not a domain and is not Noetherian.  
This ring is reduced.  So we ask: is $\W_{\ZZ_p^d}(k)$ reduced? We answer this in the case $d = 2$ and for any $p$.  

To do this, we keep track of a new partial ordering on $\mcal{F}(G)$. This new partial ordering makes sense with no extra effort for $G=\ZZ_p^d$ with any $d \geq 1$ so we state it in this generality. Consider the descending subgroup filtration
$$
G \supsetneqq pG \supsetneqq p^2G \supsetneqq \cdots \supsetneqq 
p^nG \supsetneqq p^{n+1}G \supsetneqq \cdots,
$$
which leads to the rising family of $G$-sets 
$$
0  = G/G < G/pG < G/p^2G < \cdots < G/p^nG < G/p^{n+1}G < \cdots
$$
in $\mcal{F}(G)$.  

\begin{dff}\label{dff:level}
For $G = \ZZ_p^d$, the {\it level} of $T \in \mcal{F}(G)$ is the largest $n \geq 0$ such that 
$T \geq G/p^nG$.  We write $\Lev(T) = n$.
\end{dff}

This definition makes sense, since
if $T \geq G/p^nG$  then $\#T \geq p^{nd}$, so $n$ is bounded 
above, and $T \geq 0$ so 
we have somewhere to begin. 
To get a feel for this concept, we describe it on coset spaces. 
Write $T \cong G/H$ for 
a unique open subgroup $H$ of $G$.  We have 
$T \geq G/p^nG$ if and only if $H \subset p^nG$.  
Therefore $\Lev(G/H)$ is the largest $n \geq 0$ such that 
$H \subset p^nG$. See Figure \ref{fig:Z2Lev}. 

\begin{figure}[htb]
\begin{center}
\scalebox{0.8}{
\begin{tikzpicture}[smooth]


\fill[black] (0,0) circle (0.1cm);

\fill[black] (1,0) circle (0.1cm);
\fill[black] (1,-1) circle (0.1cm);
\fill[black] (1,1) circle (0.1cm);

\fill[black] (3,1.25) circle (0.1cm);
\fill[black] (3,0.75) circle (0.1cm);

\fill[black] (3,0.25) circle (0.1cm);
\fill[black] (3,-0.25) circle (0.1cm);

\fill[black] (3,-0.75) circle (0.1cm);
\fill[black] (3,-1.25) circle (0.1cm);

\fill[black] (3,2.5) circle (0.1cm);

\fill[black] (5,3) circle (0.1cm);
\fill[black] (5,2.5) circle (0.1cm);
\fill[black] (5,2) circle (0.1cm);

\fill[black] (5,-1.2) circle (0.1cm);
\fill[black] (5,-1.45) circle (0.1cm);

\fill[black] (5,-0.65) circle (0.1cm);
\fill[black] (5,-0.9) circle (0.1cm);

\fill[black] (5,-0.15) circle (0.1cm);
\fill[black] (5,-0.4) circle (0.1cm);

\fill[black] (5,0.35) circle (0.1cm);
\fill[black] (5,0.1) circle (0.1cm);

\fill[black] (5,0.85) circle (0.1cm);
\fill[black] (5,0.6) circle (0.1cm);

\fill[black] (5,1.35) circle (0.1cm);
\fill[black] (5,1.1) circle (0.1cm);

\fill[black] (7,4.5) circle (0.1cm);

\fill[black] (7,0) circle (0.05cm);
\fill[black] (7.2,0) circle (0.05cm);
\fill[black] (7.4,0) circle (0.05cm);

\fill[black] (5.5,2.5) circle (0.05cm);
\fill[black] (5.7,2.5) circle (0.05cm);
\fill[black] (5.9,2.5) circle (0.05cm);

\fill[black] (8,5.25) circle (0.05cm);
\fill[black] (8.2,5.45) circle (0.05cm);
\fill[black] (8.4,5.65) circle (0.05cm);

\draw[-] (0,0) -- (1,0);
\draw[-] (0,0) -- (1,-1);
\draw[-] (0,0) -- (1,1);

\draw[-] (1,1) -- (3,1.25);
\draw[-] (1,1) -- (3,0.75);

\draw[-] (1,0) -- (3,0.25);
\draw[-] (1,0) -- (3,-0.25);

\draw[-] (1,-1) -- (3,-0.75);
\draw[-] (1,-1) -- (3,-1.25);

\draw[-] (-1,1.5) -- (8,1.5);

\draw[-] (1,1) -- (3,2.5);
\draw[-] (1,0) -- (3,2.5);
\draw[-] (1,-1) -- (3,2.5);

\draw[-] (3,2.5) -- (5,3);
\draw[-] (3,2.5) -- (5,2.5);
\draw[-] (3,2.5) -- (5,2);

\draw[-] (-1,3.5) -- (8,3.5);

\draw[-] (3,1.25) -- (5,1.35);
\draw[-] (3,1.25) -- (5,1.1);

\draw[-] (3,0.75) -- (5,0.85);
\draw[-] (3,0.75) -- (5,0.6);

\draw[-] (3,0.25) -- (5,0.35);
\draw[-] (3,0.25) -- (5,0.1);

\draw[-] (3,-0.25) -- (5,-0.15);
\draw[-] (3,-0.25) -- (5,-0.4);

\draw[-] (3,-0.75) -- (5,-0.65);
\draw[-] (3,-0.75) -- (5,-0.9);

\draw[-] (3,-1.25) -- (5,-1.2);
\draw[-] (3,-1.25) -- (5,-1.45);

\draw[-] (3,1.25) -- (5,3);
\draw[-] (3,0.75) -- (5,3);

\draw[-] (3,0.25) -- (5,2.5);
\draw[-] (3,-0.25) -- (5,2.5);

\draw[-] (3,-0.75) -- (5,2);
\draw[-] (3,-1.25) -- (5,2);

\draw[-] (5,3) -- (7,4.5);
\draw[-] (5,2.5) -- (7,4.5);
\draw[-] (5,2) -- (7,4.5);

\node at (-0.75,0) {Level $0$} {};
\node at (2,2.5) {Level $1$} {};
\node at (4,4.5) {Level $2$} {};

\end{tikzpicture}
}
\end{center}
\caption{Initial $\ZZ_2^2$-sets of level $0,1,2$.}
\label{fig:Z2Lev}
\end{figure}
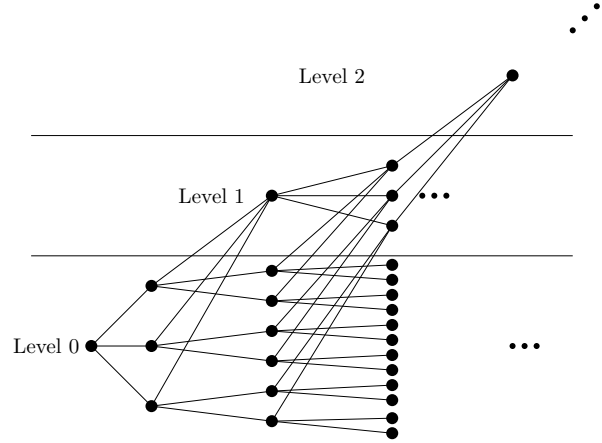


If $G = \ZZ_p$ then $\Lev(T)$ and $\#T$ are basically the same concept, since $\#T = p^{\Lev(T)}$. 
The level is something genuinely new when $G = \ZZ_p^d$ for $d \geq 2$. 
In this case neither $\#T$ nor $\Lev(T)$ determines the other; 
all we can say in general is that $\#T \geq p^{d\Lev(T)}$.

Writing the cyclic decomposition of $G/H$ as 
$\ZZ/p^{a_1}\ZZ \times \ZZ/p^{a_2}\ZZ \times \cdots \times \ZZ/p^{a_d}\ZZ$, 
its level is $\min\{a_1,a_2,\ldots,a_d\}$.  This makes it easy to produce examples.

\begin{xmp}
For $a \geq 1$, the $G$-set $\ZZ_p^d/(\ZZ_p \times p^a\ZZ_p^{d-1})$ of size 
$p^{a(d-1)}$ has level $0$ since $\ZZ_p \times p^a\ZZ_p^{d-1}$ is contained in $G$ but not $pG$. 
So when $d \geq 2$, arbitrarily large $G$-sets can have level 0 (but not when $d = 1$). 
\end{xmp}

\begin{xmp}
There are 
arbitrarily large $G$-sets of any chosen level $n$ when $d \geq 2$: 
use $\ZZ_p^d/(p^n\ZZ_p \times p^{a+n}\ZZ_p^{d-1})$ with $a \rightarrow \infty$. 
\end{xmp}

\begin{thm}
\label{thm:levelncoverby2}
For $d \geq 2$, each $\ZZ_p^d$-set of level $n$ is covered by more than one $G$-set of level $n$. 
\end{thm}
\begin{proof}
Let $T$ be a $G$-set with $\Lev(T) = n$. Pick a $\ZZ_p$-basis $\{ e_1, \ldots, e_d\}$ of $\ZZ_p^d$ so that $T \cong G/H$ with $H = \ZZ_p p^{a_1} e_1 + \ldots + \ZZ_p p^{a_d} e_d$, where $n = a_1 \leq a_2 \leq \ldots \leq a_d$. When $d \geq 3$, $\sum_{i = 1}^d \ZZ_p p^{a_i}e_i + \ZZ_p p^{a_d + 1}e_d$ and $\sum_{i \neq d-1} \ZZ_p p^{a_i}e_i + \ZZ_p p^{a_{d-1}+1}e_{d-1}$ are subgroups of $H$ with index $p$ and they are the stabilizers of two distinct $G$-sets covering $T$ both with level $n$. 
If $d = 2$, on the other hand, $\ZZ_p p^{a_1} e_1 + \ZZ_p p^{a_2 + 1}e_2$ and $\ZZ_p (p^{a_{1}+1} e_{1} + p^{a_1} e_2) +\ZZ_p p^{a_2} e_2$ are subgroups of $H$ with index $p$ and they are the stabilizers of two distinct $\ZZ_p^2$-sets covering $T$ both with level $n$. 
\end{proof}

For any $T \in \mcal{F}(G)$, $\{U : \#U \leq \#T\}$ and $\{U : U \leq T\}$ are finite 
(the latter is a subset of the former), but $\{U : \Lev(U) \leq \Lev(T)\}$ is infinite. 
This is an important distinction to remember.

\begin{remark}
\label{rmk:cyclicchain}
When $d \geq 2$, any nontrivial cyclic $G$-set looks like $\ZZ_p^d/(\ZZ_p^{d-1} \times p^a\ZZ_p)$ with $a \geq 1$
after a suitable choice of basis for $\ZZ_p^d$,  
and $\ZZ_p^{d-1} \times p^a\ZZ_p$ is not contained in $p\ZZ_p^d$ (this is false 
for $d = 1$), so all cyclic $G$-sets have level 0. When $d = 2$, a $G$-set of level $0$ is isomorphic to $\ZZ_p^2/(\ZZ_p \times p^a\ZZ_p)$, so having 
level 0 and cyclic in $\mcal{F}(\ZZ_p^2)$ mean the same thing. 
For $d \geq 3$, some $G$-sets of level 0 are not cyclic, such as 
$\ZZ_p^d/(\ZZ_p \times p\ZZ_p \times p^a\ZZ_p^{d-2})$ with $a \geq 1$.
\end{remark}

To prove that $\W_{\ZZ_p^2}(k)$ is reduced when $k$ has characteristic $p$, 
we seek the right coordinate to look at in a power of a nonzero Witt vector to know that the power is also not $\mbf{0}$.  

Say $\mbf{x} \in \W_G(k)$ and $\mbf{x} \not= \mbf{0}$.  It is natural to consider how 
$\mbf{x}$ sits in the descending ideal filtration $\{I_{p^n}\}$:
there is some $I_{p^n}$ for which $\mbf{x} \in I_{p^n}$ and $n$ is as large as possible, so 
$x_T = 0$ for $\#T < p^n$ and some $x_T$ is nonzero where $\#T = p^n$. 
In this case it is not hard to check that $\mbf{x}^2 \in I_{p^{2n}}$ (Corollary \ref{cor:IFillpower}), and we can anticipate (if $\W_G(k)$ is reduced) that 
$\mbf{x}^2$ has a nonzero coordinate at some $G$-set of size $p^{2n}$. 
Which one?  We need a way to predict a nonzero coordinate in $\mbf{x}^2$ when $\mbf{x} \not= \mbf{0}$.  


\begin{lem}
\label{lem:levelorder}
Let $G = \ZZ_p^d$. If  $U \leq T$ in $\mcal{F}(G)$ then $\Lev(U) \leq \Lev(T)$. 
\end{lem}

\begin{proof}
Let $n = \Lev(U)$, so $U \geq G/p^nG$. 
Since $T \geq U$, we have $T \geq G/p^nG$, so 
$\Lev(T) \geq n$. 
\end{proof}

Note that if $U < T$ and $d \geq 2$ it need not follow that $\Lev(U) < \Lev(T)$: we might have $\Lev(U) = \Lev(T)$. For example, if $U = G/H$ and $T = G/K$ where $H = p\ZZ_p^d = p\ZZ_p^{d-1} \times p\ZZ_p$ and $K =  p\ZZ_p^{d-1} \times p^2 \ZZ_p$ then $K \subset H$ so $U < T$ and $\Lev(U) = \Lev(T) = 1$.

\begin{lem}
\label{lem:scaling}
Let $G = \ZZ_p^d$ and $T \in \mcal{F}(G)$ have level $n$. Write $T \cong G/H$ with $H \subset p^n \ZZ_p^d$ and $H \not\subset p^{n + 1}\ZZ_p^d$ so $H = p^n \widehat{H}$ for a unique open subgroup $\widehat{H}$ of $\ZZ_p^d$. Set $\widehat{T} = G/\widehat{H}$. Then $\widehat{T}$ has level $0$. Moreover, if $U \in \mcal{F}(G)$ and $\Lev(U) = n$ then $U \leq T$ if and only if $\widehat{U} \leq \widehat{T}$. 
\end{lem}
\begin{proof}
By choosing a suitable basis of $\ZZ_p^d$, we may assume $H = p^{a_1}\ZZ_p \times \ldots \times p^{a_d}\ZZ_p$. The statement that $T$ has level $n$ means $n = \min\{a_1,\ldots,a_d\}$, so $H = p^n (p^{b_1}\ZZ_p \times \ldots \times p^{b_d}\ZZ_p)$ with some $b_i = 0$. Then $\widehat{H} = p^{b_1}\ZZ_p \times \ldots \times p^{b_d}\ZZ_p$ and $\widehat{T}$ has level $0$.

To prove the second claim, write $U \cong G/K$ and $K = p^n \widehat{K}$. We have $U \leq T$ if and only if $H \subset K$ and $\widehat{U} \leq \widehat{T}$ if and only if $\widehat{H} \subset \widehat{K}$. The conditions $H \subset K$ and $\widehat{H} \subset \widehat{K}$ are the same. 
\end{proof}

\begin{lem}
\label{lem:levcover}
For $G = \ZZ_p^d$ with $d \geq 2$ and $T \in \mcal{F}(G)$, set $\# T = p^n$. For each $m \geq n$, there exist $T'$ such that $\# T' = p^m$, $T \leq T'$ and $\Lev(T)=\Lev(T')$.
\end{lem}
\begin{proof}
Write $T = G/H$. Choose a $\ZZ_p$-basis $\{e_1,e_2,\ldots,e_d\}$ of $\ZZ_p^d$ such that $H = \sum_{i=1}^d \ZZ_p p^{a_i} e_i$. Without loss of generality, assume $a_1 \leq a_2 \leq \ldots \leq a_d$ and set $K = \sum_{i=1}^{d-1} \ZZ_p p^{a_i} e_i + \ZZ_p p^{a_d + m-n} e_d$. Then $K \subset H$ and $[H : K] = p^{m-n}$ so $G/K$ is a cover of $G/H$. Since $a_1 \leq a_2 \leq \ldots \leq a_d$ and $d \geq 2$, $\min\{a_1,a_2,\ldots,a_d\} = \min \{ a_1,a_2,\ldots, a_d + m-n\}$, so $\Lev(G/K) = \Lev(G/H)$. 
\end{proof}

Consider Figure \ref{fig:Z2Lev}. The two horizontal lines divide the diagram into regions of $\ZZ_2^2$-sets with the same level. Lemma \ref{lem:levcover} just says that these levels go infinitely far out in the frame. 

Returning to the assumption $G = \ZZ_p^2$ we have the following lemma. 

\begin{lem}
\label{lem:level1}
When $G = \ZZ_p^2$ and $T$ and $T'$ in $\mcal{F}(G)$ satisfy  $T \leq T'$ and $\Lev(T) = \Lev(T')$, $$\{ U \in \mcal{F}(G) : U \leq T' \mbox{ and } \tn{Lev}(U) = \Lev(T) \mbox{ and } \# U = \# T \} = \{T\}.$$ 
\end{lem}

\begin{proof}
The set of $G$-sets of level zero is a tree, so the property is obviously true when $\Lev(T)=\Lev(T')= 0$. Using Lemma \ref{lem:scaling} one has the result for any level. 
\end{proof}

\begin{lem}
\label{lem:nonzero}
Let $G = \ZZ_p^2$, $A$ be a nonzero commutative ring, and choose any nonzero $\mbf{a} \in W_{G}(A)$. There is 
$T_0 \in \mcal{F}(G)$ 
such that $a_{T_0} \neq 0$ and $a_U = 0$ under either of the following conditions:
\begin{itemize} 
\item $\Lev(U) < \Lev(T_0)$,
\item $\Lev(U) = \Lev(T_0)$ and $\# U < \# T_0$,
\end{itemize}
\end{lem}

\begin{proof}
Given $\mbf{a} \neq \mbf{0}$ in $\W_{G}(A)$, among $\{ T : a_T \neq 0 \}$ 
first select all $T$ with minimal level, and then among the $T$ of that minimal level, 
choose one $T$ of minimal size. Call that $T_0$.

If $\Lev(U) < \Lev(T_0)$ then $a_U = 0$ since $T_0$ is a nonzero coordinate with minimal level.
If $\Lev(U) = \Lev(T_0)$ and $\#U < \#T_0$ then $a_U = 0$ since 
otherwise $T_0$ is not a nonzero coordinate of minimal size among nonzero coordinates of minimal level.
\end{proof}

\begin{remark}
Lemma \ref{lem:nonzero} is expressed in a form convenient for the applications we have in 
mind, but it's not a result about nonzero Witt vectors so much as a property of 
$\mcal{F}(G)$: for any nonempty subset $\mcal{S}$ of $\mcal{F}(G)$, there is a $T_0 \in \mcal{S}$ 
such that $U \not\in \mcal{S}$ if $\Lev(U) < \Lev(T_0)$, or if  
$\Lev(U) = \Lev(T_0)$ and $\#U < \#T_0$.
\end{remark}

\begin{remark}
In our application of Lemma \ref{lem:nonzero} it is important to note the order in which the concepts are minimized. Here we are first choosing nonzero $G$-sets of minimal level, then among those we choose one of minimal size. These two minimizations do not commute. Figure \ref{fig:Z2MinLevel} depicts a nonzero element of $\W_{\ZZ_2^2}(k)$, where all coordinates are zero except for the circled ones which are non-zero. The $T$ coordinate is the one of minimal size first and then level, whereas the $U$ coordinate is the one of minimal level first and then size. 
\end{remark}

\begin{figure}[htb]
\begin{center}
\scalebox{0.8}{
\begin{tikzpicture}[smooth]


\fill[black] (0,0) circle (0.1cm);

\fill[black] (1,0) circle (0.1cm);
\fill[black] (1,-1) circle (0.1cm);
\fill[black] (1,1) circle (0.1cm);

\fill[black] (3,1.25) circle (0.1cm);
\fill[black] (3,0.75) circle (0.1cm);

\fill[black] (3,0.25) circle (0.1cm);
\fill[black] (3,-0.25) circle (0.1cm);

\fill[black] (3,-0.75) circle (0.1cm);
\fill[black] (3,-1.25) circle (0.1cm);

\fill[black] (3,2.5) circle (0.1cm);

\fill[black] (5,3) circle (0.1cm);
\fill[black] (5,2.5) circle (0.1cm);
\fill[black] (5,2) circle (0.1cm);

\fill[black] (5,-1.2) circle (0.1cm);
\fill[black] (5,-1.45) circle (0.1cm);

\fill[black] (5,-0.65) circle (0.1cm);
\fill[black] (5,-0.9) circle (0.1cm);

\fill[black] (5,-0.15) circle (0.1cm);
\fill[black] (5,-0.4) circle (0.1cm);

\fill[black] (5,0.35) circle (0.1cm);
\fill[black] (5,0.1) circle (0.1cm);

\fill[black] (5,0.85) circle (0.1cm);
\fill[black] (5,0.6) circle (0.1cm);

\fill[black] (5,1.35) circle (0.1cm);
\fill[black] (5,1.1) circle (0.1cm);

\fill[black] (7,0) circle (0.05cm);
\fill[black] (7.2,0) circle (0.05cm);
\fill[black] (7.4,0) circle (0.05cm);

\fill[black] (5.5,2.5) circle (0.05cm);
\fill[black] (5.7,2.5) circle (0.05cm);
\fill[black] (5.9,2.5) circle (0.05cm);

\fill[black] (4,3.25) circle (0.05cm);
\fill[black] (4.2,3.45) circle (0.05cm);
\fill[black] (4.4,3.65) circle (0.05cm);

\draw[-] (0,0) -- (1,0);
\draw[-] (0,0) -- (1,-1);
\draw[-] (0,0) -- (1,1);

\draw[-] (1,1) -- (3,1.25);
\draw[-] (1,1) -- (3,0.75);

\draw[-] (1,0) -- (3,0.25);
\draw[-] (1,0) -- (3,-0.25);

\draw[-] (1,-1) -- (3,-0.75);
\draw[-] (1,-1) -- (3,-1.25);

\draw[-] (-1,1.5) -- (6,1.5);

\draw[-] (1,1) -- (3,2.5);
\draw[-] (1,0) -- (3,2.5);
\draw[-] (1,-1) -- (3,2.5);

\draw[-] (3,2.5) -- (5,3);
\draw[-] (3,2.5) -- (5,2.5);
\draw[-] (3,2.5) -- (5,2);

\draw[-] (3,1.25) -- (5,1.35);
\draw[-] (3,1.25) -- (5,1.1);

\draw[-] (3,0.75) -- (5,0.85);
\draw[-] (3,0.75) -- (5,0.6);

\draw[-] (3,0.25) -- (5,0.35);
\draw[-] (3,0.25) -- (5,0.1);

\draw[-] (3,-0.25) -- (5,-0.15);
\draw[-] (3,-0.25) -- (5,-0.4);

\draw[-] (3,-0.75) -- (5,-0.65);
\draw[-] (3,-0.75) -- (5,-0.9);

\draw[-] (3,-1.25) -- (5,-1.2);
\draw[-] (3,-1.25) -- (5,-1.45);

\draw[-] (3,1.25) -- (5,3);
\draw[-] (3,0.75) -- (5,3);

\draw[-] (3,0.25) -- (5,2.5);
\draw[-] (3,-0.25) -- (5,2.5);

\draw[-] (3,-0.75) -- (5,2);
\draw[-] (3,-1.25) -- (5,2);

\draw (3,2.5) circle (0.15cm);

\draw (5,-1.45) circle (0.15cm);

\node at (-0.75,0) {Level $0$} {};
\node at (2,2.5) {Level $1$} {};

\node at (3,2.9) {T} {};

\node at (5.4,-1.45) {U} {};

\end{tikzpicture}
}
\end{center}
\caption{Nonzero element in $\W_{\ZZ^2_2}$}
\label{fig:Z2MinLevel}
\end{figure}


Now we will use the concept of level to prove something about multiplication in $\W_G(k)$. 
The end of the next lemma identifies a formula for a specific coordinate in the product of 
two Witt vectors if all the ``smaller'' coordinates are 0. It is analogous to something simple 
when $G = \ZZ_p$:  if $\mbf{a} = p^n(a_0 + p\mbf{a}')$ and $\mbf{b} = p^n(b_0 + p\mbf{b}')$
then $\mbf{a}\mbf{b} = p^{2n}(a_0b_0 + p\mbf{c})$.  (It is not assumed that $a_0$ and $b_0$ are nonzero.) 
There is a similar formula even if the $p$-powers in $\mbf{a}$ and $\mbf{b}$ 
are not equal, but for $G = \ZZ_p^d$ with $d \geq 2$ we don't have a formula that broad. 
It seems to be the price we pay for $\W_G(k)$ not being a domain. 

\begin{lem}
\label{lem:Fill}
Let $G$ be an abelian pro-$p$ group and $A$ be a ring of characteristic $p$. Let $V \in \mcal{F}(G)$ with $\# V \leq p^{m+n}$ for positive integers $m$ and $n$. Consider Witt vectors $\mbf{a}$ and $\mbf{b}$ in $\W_G(A)$ such that $a_U = 0$ for all $U < V$ such that $\# U < p^m$ and $b_U = 0$ for all $U < V$ such that $\# U < p^n$. Set $\mbf{c} = \mbf{ab}$. Then $c_V = 0$. 
\end{lem}

\begin{proof}
This will follow from functoriality by proving the following mod $p$ congruence for particular Witt vectors over the ring 
$R = \mbf{Z}[\underline{X},\underline{Y}]$. 
Define $\mbf{x},\mbf{y} \in \W_G(R)$ by 
\begin{displaymath}
x_T = 
\begin{cases}
0, & \text{$\# T < p^m$ and $T < V$,} \\
X_T, & \textrm{otherwise,}
\end{cases}  \text{ and } 
y_T = 
\begin{cases}
0, & \text{$\# T < p^n$ and $T < V$,} \\
Y_T, & \textrm{otherwise.}
\end{cases}
\end{displaymath} 
Set $\mbf{z} = \mbf{x}\mbf{y}$.  We will show that 
\begin{equation}\label{Ucong}
T \leq V \Longrightarrow z_T \equiv 0 \bmod pR. 
\end{equation}


Returning to the proof of (\ref{Ucong}), we argue by induction on $\#T$.
If  $\#T = 1$ then $T = 0$ and $z_0 = x_0y_0 = 0$. 
Let $p^r \leq p^{m+n}$ with $r \geq 1$ and 
assume by induction that for all $U < V$ such that $\# U < p^r$, $z_U \equiv 0 \bmod pR$. 
Pick $T \leq V$ with $\# T = p^r$.
Since the $T$-th Witt polynomial is a multiplicative function  
$W_T \colon \W_G(R) \rightarrow R$, 
$W_T(\mbf{z}) = W_T(\mbf{x})W_T(\mbf{y})$: 
$$\sum_{U \leq T} \varphi_T(U) z_U^{\# T / \#U } = \sum_{T_1,T_2 \leq T} \varphi_T(T_1)\varphi_T(T_2) x_{T_1}^{\#T / \# T_1} y_{T_2}^{\# T / \# T_2}.$$ Solving this equation for $z_T$ in $\QQ[\underline{X},\underline{Y}]$, 
\begin{equation}\label{ztf}
z_T = \sum_{T_1,T_2 \leq T} \frac{\varphi_T(T_1)\varphi_T(T_2)}{\varphi_T(T)} x_{T_1}^{\#T / \# T_1} y_{T_2}^{\# T / \# T_2} - \sum\limits_{U < T} \frac{\varphi_T(U)}{\varphi_T(T)} z_U^{\#T / \#U}.
\end{equation}

Since $z_U \in pR$ for $U < T$,
the second term in (\ref{ztf}) is $0$ mod $pR$ by Lemma \ref{lem:cong}.
In the first term  in (\ref{ztf}), 
if $T_1 = V$  then $T_1 = T = V$ and $\varphi_T(T_1)\varphi_T(T_2)/\varphi_T(T) \equiv \varphi_T(T_2) \equiv 0 \bmod pR$ by Lemma \ref{lem:propdivis} provided $T_2 \neq 0$, while if $T_2 = 0$ then $y_{T_2} = 0$. A similar argument holds if $T_2 = V$ so we can assume $T_1 < V$ and $T_2 < V$. If either $\# T_1 < p^m$ or $\# T_2 < p^n$ then $x_{T_1} = 0$ or $y_{T_2} = 0$ respectively. The remaining terms in the first sum in (\ref{ztf}) have $T_1 < V$ with $\# T_1 \geq p^m$ and $T_2 < V$ with $\# T_2 \geq p^n$. In this case $\#T_1\#T_2 \geq p^{m+n} > p^r = \#T$. Since $G$ is abelian, $\varphi_T(U) = \# U$ and so the coefficient in the first sum in (\ref{ztf}) is an integral multiple of $p$. Thus $z_T \equiv 0 \bmod pR$.     
\end{proof} 

\begin{remark}
The only place where $G$ being abelian played a role in the proof of Lemma~\ref{lem:Fill} was at the last step calculating the $p$-divisibility of the coefficients in the first sum in (\ref{ztf}). For general pro-$p$ groups, this ratio need not even be integral, however the conclusion of Lemma~\ref{lem:Fill} holds true not just for abelian pro-$p$ groups, but any pro-$p$ which satisfies $$ {\varphi_T(T_1)\varphi_T(T_2) \over \varphi_T(T)} \in p\ZZ$$ for any $T, T_1, T_2 \in \mcal{F}(G)$. 
\end{remark}

\begin{lem}
\label{lem:multFun} For $G = \ZZ_p^2$, let $\mbf{a}$ and $\mbf{b}$ be in $\W_G(k)$ such that there is a $T_0 \in \mcal{F}(G)$ such 
that $a_U = b_U = 0$ if $\Lev(U) < \Lev(T_0)$, or if $\Lev(U) = \Lev(T_0)$ and $\# U < \#T_0$. Set $\#T_0 = p^n$. For any $T \in \mcal{F}(G)$ with size $p^{2n}$ such that $T_0 \leq T$ and $\Lev(T_0) = \Lev(T)$, the product $\mbf{a}\mbf{b}$ has $T$-coordinate $(a_{T_0}b_{T_0})^{p^n}$.  
\end{lem}

\begin{proof}
Let $R = \mbf{Z}[\underline{X},\underline{Y}]$. Define $\mbf{x},\mbf{y}$ in $\W_G(R)$ by 
\begin{displaymath}
x_U = 
\begin{cases}
0 & \text{if $\Lev(U) < \Lev(T_0)$,} \\ 
0 & \text{if $\Lev(U) = \Lev(T_0)$ and $\#U < p^n$,} \\
X_U & \textrm{otherwise,}
\end{cases}
\end{displaymath} and 
\begin{displaymath}
y_U = 
\begin{cases}
0 & \text{if $\Lev(U) < \Lev(T_0)$,} \\
0 & \text{if $\Lev(U) = \Lev(T_0)$ and $\# U < p^n$,} \\
Y_U & \textrm{otherwise.}
\end{cases}
\end{displaymath} 
and set $\mbf{z} = \mbf{xy}$. For any $G$-set $T$ of size $p^{2n}$ with $T_0 \leq T$ and $\Lev(T_0) = \Lev(T)$ we will show $z_T \equiv (X_{T_0}Y_{T_0})^{p^n} \bmod pR.$ By functoriality the lemma would then follow. 

For any $G$-set $T$, 
$W_T(\mbf{z}) = W_T(\mbf{x})W_T(\mbf{y})$:
\begin{equation}
\label{eqprod}
\sum\limits_{U \leq T} \# U z_U^{\#T/\#U} = \left(\sum_{U \leq T} \#U x_U^{\#T/\#U}\right)\left(\sum_{U \leq T} \#U y_U^{\#T/\#U}\right) .
\end{equation}

By Lemma \ref{lem:levcover} we can choose $T \geq T_0$ such that $\# T = p^{2n}$ and $\Lev(T)  = \Lev(T_0)$.
First we look at the left side of (\ref{eqprod}).  
Let $V < T$, so $\# V < p^{2n}$.
By Lemma \ref{lem:levelorder} $\Lev(V) \leq \Lev(T) = \Lev(T_0)$. So either $\Lev(V) < \Lev(T)$ or $\Lev(V) = \Lev(T)$. For all $U < V$ with $\# U < p^n$, both $x_U$ and $y_U$ are zero by hypothesis. The proof of Lemma \ref{lem:Fill} tells us that $z_V \equiv 0 \bmod pR$ when $V < T$ and $\# V < p^{2n}$.
So $\# V z_V^{\#T/\#V} \equiv 0 \bmod p^{2n+1}R$ for $V < T$ and $\# V < p^{2n}$, so the left side of (\ref{eqprod}) is 
$p^{2n}z_T \bmod p^{2n+1}R$. 

Now we turn to the right side of (\ref{eqprod}). 
If $U \leq T$ then $\Lev(U) \leq \Lev(T)$ (Lemma \ref{lem:levelorder}) and since $\Lev(T) = \Lev(T_0)$
our hypotheses tells us $x_U = 0$ and $y_U = 0$ if $\Lev(U) < \Lev(T)$,
or if $\Lev(U) = \Lev(T)$ and $\#U < p^n$. 
So the only $U$-terms in each sum on the right side of (\ref{eqprod}) that 
are not automatically 0 have $\Lev(U) = \Lev(T)$ and $\#U \geq p^n$.

Dropping the terms in (\ref{eqprod}) where $x_U = 0$ and $y_U = 0$ and reducing modulo $p^{2n+1}R$, one has that $p^{2n}z_T$ is equivalent to 
$$
\left(\sum_{\stackrel{\stackrel{U \leq T}{\Lev(U)=\Lev(T)}}{\# U \geq p^n}} \#U X_U^{\#T/\#U}\right)\left(\sum_{\stackrel{\stackrel{U \leq T}{\Lev(U)=\Lev(T)}}{\# U \geq p^n}}  \#U Y_U^{\#T/\#U}\right)  \bmod p^{2n+1}R.
$$ 
Each summand in the two sums has a coefficient divisible at least by $p^n$, so any product 
of a term from each sum is divisible by $p^{2n}$.  A term divisible by $p^{n+1}$ in one sum 
has a product with any term in the other sum that is $0 \bmod p^{2n+1}R$, so
$$
p^{2n}z_T  \equiv 
\left(\sum_{\stackrel{\stackrel{U \leq T}{\Lev(U)=\Lev(T)}}{\# U = p^n}} p^n X_U^{p^n}\right)\left(\sum_{\stackrel{\stackrel{U \leq T}{\Lev(U)=\Lev(T)}}{\# U = p^n}}  p^nY_U^{p^n}\right)  \bmod p^{2n+1}R.
$$

Dividing through by $p^{2n}$ and using additivity of the $p$th power map in $R/pR$, 
\begin{equation}
\label{eq:multFun1}
z_T  \equiv 
\left(\sum_{\stackrel{\stackrel{U \leq T}{\Lev(U)=\Lev(T)}}{\# U = p^n}}X_U \cdot \sum_{\stackrel{\stackrel{U \leq T}{\Lev(U)=\Lev(T)}}{\# U = p^n}}  Y_U\right)^{p^n}  \bmod pR.
\end{equation}

What are the $U$ lying below $T$ in $\mcal{F}(G)$ with the same level as $T$ and of size $p^n$? 
One example is $U = T_0$.  By Lemma \ref{lem:level1} this is the only example, so $z_T \equiv (X_{T_0}Y_{T_0})^{p^n} \bmod pR$. 
\end{proof}

We will only apply Lemma \ref{lem:multFun} when the two Witt vectors are equal, i.e., to the square of a nonzero element of $\W_G(k)$. 

\begin{remark}
Most of the proof of Lemma \ref{lem:multFun} goes through for $G = \ZZ_p^d$ for $d \geq 3$ and not just $d = 2$. In fact (\ref{eq:multFun1}) is true for $G = \ZZ_p^d$ for $d \geq 2$ and it was only at the last step where we used the fact that $d = 2$, which hinged on Lemma \ref{lem:level1}. Lemma \ref{lem:level1} is not true for $G = \ZZ_p^d$ with $d \geq 3$ since the level $0$ part of $\mcal{F}(\ZZ_p^d)$ is not a tree. For example when $d \geq 3$, $\ZZ_p^d/(\ZZ_p \times p\ZZ_p \times p^a\ZZ_p^{p-2})$ with $a \geq 1$ is not a cyclic group and so its strict downset is not a chain. 
\end{remark}

\begin{thm}
\label{thm:Red}
For any field $k$ of characteristic $p$, the ring $\W_{G}(k)$ is reduced for $G = \ZZ_p^2$.
\end{thm}

\begin{proof}
Let $\mbf{v}$ be nonzero in $\W_{G}(k)$. If $v_0 \neq 0$ then $\mbf{v} \in \W_{G}(k)^\times$ by Theorem \ref{units}. Otherwise $\mbf{v} \in \goth{m}$ and it suffices to show for all $n \in \mbf{Z}^+$ that $\mbf{v}^{2^n}$ is nonzero. Thus it suffices to show if $\mbf{v} \in \goth{m}$ and $\mbf{v} \neq \mbf{0}$ then $\mbf{v}^2 \neq \mbf{0}$.  By Lemma \ref{lem:nonzero} there is a 
$T_0 \in \mcal{F}(G)$ such that $v_{T_0}$ is nonzero but $v_U$ is zero for all 
$U \in \mcal{F}(G)$ where $\Lev(U) < \Lev(T_0)$ or where 
$\Lev(U) = \Lev(T_0)$ and $\#U < \#T_0$.  Lemma \ref{lem:multFun} 
with $\mbf{a} = \mbf{b} = \mbf{v}$ tells us that $\mbf{v}^2$ has a coordinate equal to $v_{T_0}^{2\#T_0} \not= 0$, so $\mbf{v}^2 \not= \mbf{0}$.
\end{proof}

\section*{Acknowledgments}
This paper is based on the author's Ph.D. thesis at the University of Connecticut \cite{Thesis}. The author is forever indebted to his advisor Keith Conrad, without whose endless patience and support this work would not have been possible. The author would also like to thank Sarah Glaz, \'{A}lvaro Lozano-Robledo, and Jeremy Teitelbaum for many helpful and enlightening discussions and editorial suggestions, and to Anurag Singh, Jesse Elliott and Karl Schwede for their multiple readings and helpful comments.

\Addresses
\end{document}